\pdfoutput=1
\documentclass[11pt]{amsart}
\usepackage{amssymb}
\usepackage{amsmath}
\usepackage[hidelinks]{hyperref}
\usepackage{etoolbox}
\usepackage{enumitem}
\usepackage{xcolor}
\usepackage{tikz-cd}
\usepackage[normalem]{ulem}

\newtheorem{theorem}{Theorem}[section]
\newtheorem{lemma}[theorem]{Lemma}

\newtheorem{proposition}[theorem]{Proposition}

\newtheorem{definition}[theorem]{Definition}
\newtheorem{corollary}[theorem]{Corollary}
\newtheorem{conjecture}[theorem]{Conjecture}
\newtheorem{problem}[theorem]{Problem}

\theoremstyle{remark}
\newtheorem{remark}[theorem]{Remark}
\newtheorem{example}[theorem]{Example}

\newcommand{\p}{\vskip .4cm}

\newcommand{\F}{\mathbb{F}}
\newcommand{\Z}{\mathbb{Z}}
\newcommand{\Q}{\mathbb{Q}}
\newcommand{\R}{\mathbb{R}}
\newcommand{\C}{\mathbb{C}}
\newcommand{\Cc}{\mathbb{C}^{\times}}

\newcommand{\cA}{\mathcal{A}}
\newcommand{\cB}{\mathcal{B}}
\newcommand{\cF}{\mathcal{F}}

\newcommand{\cO}{\mathcal{O}}

\newcommand{\sG}{\mathsf{G}}
\newcommand{\sH}{\mathsf{H}}
\newcommand{\sS}{\mathsf{S}}
\newcommand{\sT}{\mathsf{T}}
\newcommand{\Gm}{\mathbb{G}_m}

\newcommand{\til}{\tilde}
\newcommand{\bsl}{\backslash}
\newcommand{\ra}{\rightarrow}
\newcommand{\ira}{\hookrightarrow}
\newcommand{\sra}{\twoheadrightarrow}
\newcommand{\xra}{\xrightarrow}

\newcommand{\Lie}{\operatorname{Lie}}
\newcommand{\rank}{\operatorname{rank}}
\newcommand{\depth}{\operatorname{depth}}
\newcommand{\Ad}{\operatorname{Ad}}
\newcommand{\ad}{\operatorname{ad}}

\newcommand{\Irr}{\operatorname{Irr}}

\newcommand{\Tr}{\operatorname{Tr}}

\newcommand{\fg}{\mathfrak{g}}
\newcommand{\fh}{\mathfrak{h}}
\newcommand{\fp}{\mathfrak{p}}
\newcommand{\fu}{\mathfrak{u}}
\newcommand{\fl}{\mathfrak{l}}

\newcommand{\ft}{\mathfrak{t}}
\newcommand{\fm}{\mathfrak{m}}
\newcommand{\fz}{\mathfrak{z}}

\newcommand{\matr}[1]{\left[\begin{matrix}#1\end{matrix}\right]}
\newcommand{\inn}[2]{\langle#1,#2\rangle}

\newcommand{\slt}{\mathfrak{sl}_2}

\newcommand{\supp}{\operatorname{supp}}
\newcommand{\dep}{\operatorname{dep}}

\newcommand{\WF}{\operatorname{WF}}

\newcommand{\DB}{\operatorname{DB}}

\begin{document}
	
	\title{Wave-front sets for $p$-adic Lie algebras}
	\author{Cheng-Chiang Tsai}
	\thanks{The author is supported by NSTC grant 112-2115-M-001-009.}
	\email{chchtsai@gate.sinica.edu.tw}
	\address{Academia Sinica, Institute of Mathematics, 6F, Astronomy-Mathematics Building, No. 1, Sec. 4, Roosevelt Road, Taipei 106319, Taiwan, and\vskip.2cm
		National Sun Yat-Sen University, Department of Applied Mathematics, No. 70, Lienhai Rd., Kaohsiung 80424, Taiwan}
	\begin{abstract}
		We study the wave-front set of an element in a $p$-adic reductive Lie algebra (for $p\gg\rank$), namely the set of maximal nilpotent orbits appearing in its Shalika germ expansion. By adapting an algorithm of Waldspurger that computes orbital integrals, we obtain an inductive algorithm to compute an invariant that determines the wave-front set. This gives an algorithm to compute the wave-front sets for regular supercuspidal representations and reveals examples whose wave-front sets are not contained in a single geometric orbit, for arbitrarily large $p$ within a fixed rank.
	\end{abstract}
	\makeatletter
	\patchcmd{\@maketitle}
	{\ifx\@empty\@dedicatory}
	{\ifx\@empty\@date \else {\vskip3ex \centering\footnotesize\@date\par\vskip1ex}\fi
		\ifx\@empty\@dedicatory}
	{}{}
	\patchcmd{\@adminfootnotes}
	{\ifx\@empty\@date\else \@footnotetext{\@setdate}\fi}
	{}{}{}
	\makeatother

	\maketitle 
	
	\vskip-1cm $\;$
	
	\tableofcontents
	
	\vskip-1cm $\;$
	
	\section{Introduction}
	
	Let $F$ be a non-archimedean local field with residue field $k$ and $p:=\operatorname{char}(k)$, $G$ a connected reductive group over $F$ and $\fg:=\Lie G$. Throughout this paper we enforce some restrictions on $p:=\mathrm{char}(k)$ for which we refer to Appendix \ref{app:char}. In the study of representations and characters of $G(F)$, it has been long observed that they demonstrate an inductive structure. For the more recent landmarks, Jiu-Kang Yu \cite{Yu01} constructed supercuspidal representations inductively from smaller twisted Levi subgroups. In \cite{KM06}, Ju-Lee Kim and Fiona Murnaghan gave an inductive proof for a formula of the local characters of these and more representations. After that, Jeff Adler and Loren Spice \cite{AS09} and eventually Spice \cite{Spi18,Spi21} obtained an inductive description for the full character of Yu's supercuspidal representations, which are known to be all supercuspidal representations \cite{Kim07,Fin21} under restrictions on $p$. Meanwhile, Jean-Loup Waldspurger \cite{Wa06} proved that the fundamental lemma in mixed and equal characteristic are equivalent, by developing an inductive algorithm to compute the orbital integrals needed and showing that the algorithm depends only on some data defined in terms of the residue field. (This was already done for $GL_n$ in \cite{Wal91}.) In \cite{Tsa15} the author mentioned how a pretty similar algorithm can be used to compute Shalika germs. All aforementioned works largely employ the same inductive process which we discuss in \S\ref{sec:RedtoLevi}.
	
	Let us recall the theory of Shalika germs as enhanced by the homogeneity theorems of Waldspurger \cite{Wa95} and DeBacker \cite{De02a}. Denote by $\fg^{nil}(F)/\!\!\sim$ the set of nilpotent $\Ad(G(F))$-orbits in $\fg(F)$. We have
	
	\begin{theorem}\label{thm:DB} \cite[Thm. 2.1.5]{De02a} For any $\gamma\in\fg(F)$, there exists $c_{\cO}(\gamma)\in\Q$ indexed by $\cO\in\fg^{nil}(F)/\!\!\sim$ and a lattice $\Lambda_{\gamma}\subset\fg(F)$ such that
		\[
		I_{\gamma}(f)=\sum_{\cO\in\fg^{nil}(F)/\sim}c_{\cO}(\gamma)I_{\cO}(f),\;\forall f\in C_c(\fg(F)/\Lambda_{\gamma}).
		\]
	\end{theorem}
	\noindent Here we denote by $C_c(\fg(F)/\Lambda_{\gamma})\subset C_c^{\infty}(\fg(F))$ the space of $\C$-valued compactly supported functions satisfying $f(X+Y)=f(X)$ for any $Y\in\Lambda_{\gamma}$. The coefficients $c_{\cO}(\gamma)$ are called {\bf Shalika germs}, generalizing the theorem of Shalika \cite{Sh72} for $\gamma\in G(F)$ strongly regular semisimple and sufficiently close to $1$. Consider the partial order on  $\fg^{nil}(F)/\!\!\sim$ for which $\cO_1\le\cO_2$ iff the former is contained in the closure of the latter in the $p$-adic topology. With Theorem \ref{thm:DB}, we may define the wave-front set
	\[
	\WF(\gamma):=\max\{\cO\in\fg^{nil}(F)/\!\!\sim\;|\;c_{\cO}(\gamma)\not=0\}.
	\] 
	This can be viewed as the linearized version of wave-front sets for representations of $G(F)$. For example, the wave-front set for a regular supercuspidal representation \cite{Kal19} is equal to the aforementioned wave-front set for specific elliptic $\gamma\in\fg(F)$, see Remark \ref{rmk:wfsc}. Given that wave-front sets of representations of $G(F)$ are important for degenerate Whittaker models \cite{MW87} and related to the Gan-Gross-Prasad program (see e.g. \cite[\S4]{JLZ22a}), $\WF(\gamma)$ appears to be a basic object for the representation theory of $G(F)$, and a natural question is whether we can compute $\WF(\gamma)$ inductively.
	
	It came as a surprise to us that if we take the question in the literal and strict sense, then the answer is no. Namely in the inductive process when we write $\gamma=\gamma_r+\gamma_{>r}$ and $H:=Z_G(\gamma_r)$, the set $\WF^G(\gamma)$ {\bf cannot} be determined by $\gamma_r$ and $\WF^H(\gamma_{>r})$; see Example \ref{ex:WFnotinductive}. This is the subtlety discussed in \cite[Rmk. 3.6]{Nev23}. Nevertheless, we define for any $s\in\R$ another set (see \S\ref{sec:DB} for the notation):
	\[
	c_{G,s}(\gamma):=\{(\varsigma,c)\;|\;\varsigma\text{ an }s\text{-facet, }c\in\fg(F)_{\varsigma=s},\;\Ad(G(F))\varsigma\cap\left(c+\fg(F)_{\varsigma>s}\right)\not=\emptyset\}.
	\]
	The set implicitly appeared in \cite{Wa06} and also resembles the residue field version of $\operatorname{Nil}(\gamma)$ in \cite{Nev23}. The set $c_{G,s}(\gamma)$ determines $\WF(\gamma)$, and we manage to adapt the algorithm of Waldspurger in \cite{Wa06} to inductively compute $c_{G,r}(\gamma)$. Consider
	\begin{equation}\label{eq:inIntro}
	C_{G,s}:=\{(\varsigma,c)\;|\;\varsigma\text{ an }s\text{-facet, }c\in\fg(F)_{\varsigma=s}\}
	\end{equation}
	so that every $c_{G,s}(\gamma)$ is a subset of $C_{G,s}$. There will be a set $C_{G}$ which is an appropriate union $C_{G}=\bigcup_{s\in\R}C_{G,s}$ and we will define a directed acyclic graph with vertex set $C_{G}$ and some edges, so that $c_{G,s}(\gamma)$ is equal the subset of those vertices $C_{G}$ who have a path to any vertex in $c_{G,r}(\gamma)$. On the other hand, when $\gamma=\gamma_r+\gamma_{>r}$, $c_{G,r}(\gamma)$ is explicitly determined by $c_{H,r}(\gamma_{>r})$ where $H=Z_{G}(\gamma_r)$. Together this gives an inductive algorithm to determine any $c_{G,s}(\gamma)$.

	There are two kinds of complications for this algorithm. For a vertex $v$ in the graph $C_G$, the set of its out-neighbors correspond to the set of rational orbits in specific graded Lie algebra over $k$ that intersect specific affine subspace (e.g. Slodowy slices). A good theory about such orbits - and therefore the edges of the graph - seems very difficult to achieve. We call this the geometric problem (Problem \ref{prob:hard}). On the other hand, the graph $C_G$ is defined in terms of the Bruhat-Tits building for $G$, or effectively an apartment of it, which is still sort of complicated. Furthermore, for a given directed acylic graph it needs not be easy to determine if there is a path from a specific vertex to another set of vertices (``{\it does this maze have a solution?}''). We call these the combinatorial problem.
	
	We prove some necessary conditions and sufficient conditions for edges in the graph to exist, mostly to address the geometric problem. Our main existence result, Corollary \ref{cor:goodness}, is based on Theorem \ref{thm:good1} which generalizes the result established by Dan Ciubotaru and Emile Okada in \cite[\S4]{CO23}. When the semisimple part of $\gamma$ is contained in $(\Lie T)(F)$ for $T\subset G$ an unramified subtorus, it seems possible that necessary and sufficient conditions are already strong enough to give the correct unramified wave-front set $\WF^{ur}(\gamma)$; see Conjecture \ref{conj:upperbound}. In some cases we analyze our algorithm and confirm that. This includes the case corresponding to unramified toral supercuspidal representations already worked out in \cite{CO23}, as well as an interesting example (Example \ref{ex:u6}) with $G=U_6$ and that $\WF(\gamma)$ contains orbits of type $[4,1,1]$ and $[3,3]$, for $p$ arbitrarily large. 
	
	Let us now explain the organization of the paper. In \S\ref{sec:DB} we explain how to determine $\WF(\gamma)$ from $c_{G,s}(\gamma)$. In \S\ref{sec:RedtoLevi} we explain a result of Kim-Murnaghan that allows us to reduce the computation of $c_{G,s}(\gamma)$ to twisted Levi subgroups. Then in \S\ref{sec:graph} we present the main construction (see Theorem \ref{thm:graph}) of the graph $C_G$, hence our algorithm for computing $c_{G,s}(\gamma)$, citing a construction of Waldspurger to be explained in \S\ref{sec:geometric}. The longest section \S\ref{sec:geometric} analyzes the algorithm, or mostly analyzes a number of questions arising from \cite{De02b} relating Moy-Prasad quotients and nilpotent orbits. Several results in \S\ref{sec:geometric} will refer to Appendix \ref{sec:gradedSpringer} regarding (graded) Lie algebras over the residue field. In \S\ref{sec:unram} we present demonstrations and examples for the algorithm, and lastly in \S\ref{sec:rep} we discuss the relation between wave-front sets for elements in the Lie algebra and wave-front sets for representations of $G(F)$.
	
\subsection*{Acknowledgment} I wish to express my profound thanks to Monica Nevins who generously explained to me a subtlety in \cite[Thm 3.5]{Nev23} and kindly shared with me an earlier version of \cite{Nev23}. I benefit a lot from communication with Chi-Heng Lo about the structure of wave-front sets, for which I deeply appreciate. I am very grateful to Emile Okada for many illuminating discussions, particularly his instructive explanation about the theory of $G$-sheets.
I also owe many thanks to Ju-Lee Kim, Jessica Fintzen, Loren Spice, Stephen DeBacker and Masao Oi for numerous inspiring discussions. I am indebted to the organizers of the 5th Nisyros Conference on Automorphic Representations
\& Related Topics where many of these wonderful discussions happened, and the same for the 2023 Japan-Taiwan Joint Conference on Number Theory and the University of Bonn.
	
\section{DeBacker's parametrization}\label{sec:DB}

We begin with some notations. Let $\cB(G,F)$ be the Bruhat-Tits building (hereafter always the extended building) for $G/_F$. For any $x\in\cB(G,F)$, $r\in\R$, denote by $\fg(F)_{x\ge r}$ the Moy-Prasad lattice \cite[\S3]{MP94} of depth $r$, by $\fg(F)_{x>r}$ that of depth $>r$, and by $\fg(F)_{x=r}:=\fg(F)_{x\ge r}/\fg(F)_{x>r}$ the Moy-Prasad quotient at depth $r$. We also denote by $G(F)_{x\ge 0}$ the parahoric subgroup and $\sG_x/_k$ the (algebraic) reductive quotient. In particular $\sG_x(k)$ is a quotient of $G(F)_{x\ge 0}$ with pro-$p$ kernel, and $\Lie\sG_x=\fg(F)_{x=0}$.

\begin{remark}\label{rmk:MP} Our notations for the Moy-Prasad filtration are different from the usual notations. The new notations are extensively used in \cite{Spi18} and \cite{Spi21}. They will also prove convenient for our purpose here.
\end{remark}

We also denote by $\fg(F)_{\ge r}:=\bigcup_{x\in\cB(G,F))}\fg(F)_{x\ge r}$ and likewise $\fg(F)_{>r}:=\bigcup_{x\in\cB(G,F))}\fg(F)_{x>r}$. We define 
\begin{equation}
	\dep(\gamma):=\max\{r\in\R\;|\;\gamma\in\fg(F)_{\ge r}\}
\end{equation}
Following \cite{De02b} we say $x_1,x_2\in\cB(G,F)$ are in the same $r$-facet iff $\fg(F)_{x_1\ge r}=\fg(F)_{x_2\ge r}$ and $\fg(F)_{x_1>r}=\fg(F)_{x_2>r}$. 
Denote by $\Phi^G(r)$ the set of $r$-facets, which form a partition of $\cB(G,F)$. Recall that $\fg^{nil}(F)\subset\fg(F)$ is the nilpotent cone. Let us begin with \cite[Def. 4.1.1]{De02b}:

\begin{definition}\label{def:nil} Let $\varsigma\in\Phi^G(r)$ for some $r\in\R$. We say a coset $c\in\fg(F)_{\varsigma=r}$ is {\bf nilpotent} iff $(c+\fg(F)_{\varsigma>r})\cap\fg^{nil}(F)\not=\emptyset$.
\end{definition}

In {\it loc. cit.} DeBacker (following Moy-Prasad) call such cosets degenerate. We choose the term ``nilpotent'' here, since it relates very well with other usage of nilpotency (but degeneracy can mean several different things). We also remark that we tend to think of $c$ as if it was an element in $\fg(F)$. This indeed can be realized if $F$ is a local function field, and that the coset is nilpotent iff it is nilpotent in $\fg(F)$. In any case we will usually write $c+\fg(F)_{\varsigma>r}$ (instead of just $c$) to emphasize the coset as a subset of $\fg(F)$. Recall the notation $C_{G,r}=\{(\varsigma,c)\;|\;\varsigma\in\Phi^G(r),c\in\fg(F)_{\varsigma=r}\}$. Let us also write
\[ C_{G,r}^{nil}=\{(\varsigma,c)\in C_{G,r}\;|\;c\text{ is nilpotent}\}.
\]

For any $(\varsigma,c)\in C_{G,r}\in C_{H,r}^{nil}$, \cite[Cor. 5.2.5, Lemma 5.3.3]{De02b} asserts the existence of a unique minimal nilpotent orbit $\cO(\varsigma,c)$ such that $c+\fg(F)_{\varsigma>r}\cap\cO(\varsigma,c)\not=\emptyset$. 
Whenever $C\subset C_{G,r}^{nil}$ is a subset, we denote by
\[
\DB_r(C)=\{\cO(\varsigma,c)\;|\;(\varsigma,c)\in C\}.
\]
For two subsets $S_1,S_2\in\fg^{nil}(F)/\!\!\sim$ we say $S_1\equiv S_2$ iff, for every $\cO_i\in S_i$ there exists $\cO_j\in S_j$ such that $\cO_i\le \cO_j$, for any $i,j=1,2$. We have
	
\begin{proposition}\label{prop:wf2} Let $\gamma\in\fg(F)$ with $\dep(\gamma)=r$. For any $s<r$ we have $\DB(c_{G,s}(\gamma))\equiv\WF(\gamma)$.
\end{proposition}

\begin{proof} Suppose $\cO_0\in \DB_s(c_{G,s}(\gamma))$. By definition this says there exists $\varsigma\in\cF(\gamma)$ and $c\in\fg(F)_{\varsigma=s}$ such that $\gamma\in c+\fg(F)_{\varsigma>s}$ and $\cO_0=\cO(\varsigma,c)$. Denote by $[c+\fg(F)_{\varsigma>s}]$ the characteristic function of the coset. Since $\dep(\gamma)>s$, by Theorem \cite[Thm. 2.1.5]{De02a} we have
	\begin{equation}\label{eq:useDeB02a}
	I_{\gamma}([c+\fg(F)_{\varsigma>s}])=\sum_{\cO\in\fg^{nil}(F)/\sim}c_{\cO}(\gamma)I_{\cO}([c+\fg(F)_{\varsigma>s}])
	\end{equation}
	We have $I_{\gamma}([c+\fg(F)_{\varsigma>s}])>0$. By definition of $\cO(\varsigma,c)$, we also have $I_{\cO}([c+\fg(F)_{\varsigma>s}])>0$ iff $\cO\ge\cO(\varsigma,c)$. Hence we have $c_{\cO}(\gamma)\not=0$ for some $\cO\ge\cO(\varsigma,c)$.
	
	To conclude that $\DB_s(c_{G,s}(\gamma))\equiv\WF(\gamma)$, suppose on the contrary we have $\cO_0\in\WF(\gamma)$ such that $\cO_0\not\le\cO(\varsigma,c)$ for any $(\varsigma,c)\in c_{G,s}(\gamma)$. By \cite[Thm. 5.6.1]{De02b} we can write $\cO_0=\cO(\varsigma',c')$ for some $(\varsigma',c')\in C_{G,s}$. Applying the last equation to $[c'+\fg(F)_{\varsigma'>s}]$, we see that the LHS of \cite{De02a} vanishes while the RHS does not. This gives the desired contradiction and proves the proposition.
\end{proof}

\section{Reduction to twisted Levi}\label{sec:RedtoLevi}

In this section we fix
\begin{enumerate}
	\item A semisimple element $\gamma_r$ that is {\bf good} of depth $r$ in the sense of \cite[Def. 5.2]{AR00}. That is, for any torus $T\subset G$ such that $\gamma_r\in\Lie T$, and for every $\alpha\in\Phi(G,T)$, not necessarily defined over $F$, we have that $d\alpha(\gamma_r)$ is either $0$ or has normalized valuation $r$, where a uniformizer of $F$ has valuation $1$. We write $H:=Z_G(\gamma_r)$ and $\fh=\Lie H$.
	\item An element $\gamma_{>r}\in\fh(F)_{>r}$.
\end{enumerate}
We write $\gamma:=\gamma_r+\gamma_{>r}$. Such a decomposition is always possible for any $\gamma$ and $r=\dep(\gamma)$ by applying \cite[Prop. 5.4]{AR00} to the semisimple part of $\gamma$, and such that $H\subsetneq G$ is a proper twisted Levi subgroup.
There is an embedding $\cB(H,F)\ira\cB(G,F)$ that is canonical up to the action of $X_*(Z(H)/_F)\otimes_{\Z}\R$ on $\cB(H,F)$. Here $X_*(Z(H)/_F)$ is the group of cocharacters $\mathbb{G}_m\ra Z(H)$ defined over $F$. We fix any such embedding $\cB(H,F)\ira\cB(G,F)$.

The key step is to apply \cite[Lemma 2.3.3]{KM03} with their $\Gamma$ being our $\gamma_r$, their $X'$ being our $\gamma_{>r}$, their $G'$ being our $H$, and their $d(\Gamma)$ being our $r$. The reference asserts that for any $x\in\cB(G,F)$:
\begin{equation}\label{eq:KM}
	\gamma_r\in\fg(F)_{x\ge r}\Leftrightarrow x\in\cB(H,F).
\end{equation}
\begin{equation}\label{eq:KM2}
	\gamma_r+\gamma_{>r}\in\fg(F)_{x\ge r}\Rightarrow x\in\cB(H,F).
\end{equation}
Note that (\ref{eq:KM}) has the following nice application:
\begin{corollary}\label{cor:KM} For an $r$-facet $\varsigma\subset\cB(G,F)$ for $G$ which intersects $\cB(H,F)$ non-trivially, we have $\varsigma\subset\cB(H,F)$. In other words, $\cB(H,F)$ is the union of some $r$-facets in $\cB(G,F)$.
\end{corollary} 

Namely every $\varsigma\in\Phi^G(r)$ is contained in an $r$-facet in $\cB(H,r)$, which we denote by $\varsigma^H\in\Phi^H(r)$. Let $\gamma_{=r}$ be the image of $\gamma_r$ in $Z(\fh)(F)_{=r}$. We will derive from (\ref{eq:KM}) and  (\ref{eq:KM2}) that

\begin{proposition}\label{prop:redtoLevi} The set $c_{G,r}(\gamma)$ is determined by $c_{H,r}(\gamma_{>r})$ is the following identity
	\begin{equation}\label{eq:descent}
		c_{G,r}(\gamma)=G(F).\{(\varsigma,\,c_H+\gamma_{=r})\;|\;\varsigma\in\Phi^G(r),(\varsigma^H,c_H)\in c_{H,r}(\gamma_{>r})\}.
	\end{equation}
\end{proposition}

\begin{proof}
For the ``$\subset$'' direction, say we have $(\varsigma,c)\in c_{G,r}(r)$. Modulo $G(F)$-action we may WLOG assume $\gamma\in c+\fg(F)_{\varsigma>r}\subset \fg(F)_{\varsigma\ge r}$. By (\ref{eq:KM2}), this gives $\varsigma\subset\cB(H,F)$. Since $\gamma\in\fh(F)$, this shows that $c=c_H$ for some $c_H\in\fh(F)_{\varsigma^H=r}$, and $\gamma_{>r}\in c_H-\gamma_{=r}+\fh(F)_{\varsigma^H=r}$. The latter implies $(\varsigma^H,c_H-\gamma_{=r})\in \underline{c}_{\fh,r}(\gamma_{>r})$, which by definition says $(\varsigma,c)$ is contained in the RHS of (\ref{eq:descent}).

For the ``$\supset$'' direction, let $(\varsigma,c)\in C_{G,r}$ be such that there exists $g\in G(F)$ such that $g^{-1}.\varsigma\subset\cB(H,F)$ and there exists $c_H$ such that
\begin{enumerate}
	\item $(\varsigma^H,c_H)\in c_{H,r}(\gamma_{>r})$, and 
	\item $g^{-1}.c=\gamma_{=r}+c_H\in\fh(F)_{\varsigma^H=r}$. 
\end{enumerate} 
That $c_H\in c_{H,r}(\gamma_{>r})$ is to say that there exists $h\in H(F)$ such that $\Ad(h)\gamma_{>r}\in c_H+\fh(F)_{\varsigma^H>r}$. Since $\gamma_r\in\gamma_{=r}+Z(\fh)(F)_{>r}\subset Z(\fh)_{\ge r}$, this implies
\[
\Ad(g^{-1}h)(\gamma)=\Ad(g^{-1})\Ad(h)\gamma_{>r}+\Ad(g^{-1})\gamma_r\in c_H+\gamma_{=r}+ \fh(F)_{\varsigma^H>r}\subset c+\fg(F)_{\varsigma>r}
\]
which gives $(\varsigma,c)\in c_{G,r}(\gamma)$.
\end{proof}

\section{The graph $C_{G}$}\label{sec:graph}

The goal of this section is to define the set $C_{G}$ ($C$ for cosets)  as a glued union $C_{G}=\bigcup_{r\in\R} C_{G,r}$ and a graph with vertex set $C_{G}$ which we will use to compute $c_{G,r}(\gamma)$. In fact, we will allow some choices in defining the graph, yet any set of choices produce a graph that satisfies Theorem \ref{thm:graph} below. A large part of this section comes from \cite{Wa06}, especially \S2.4 and \S7 {\it op. cit.}. 
Recall that $\Phi^G(r)$ is the set of $r$-facets and $C_{G,r}:=\{(\varsigma,c)\;|\;\varsigma\in\Phi^G(r),\;c\in\fg(F)_{\varsigma=r}\}$.
Note that $C_{G,r}$ is essentially the sets $I_r$ in the introduction of \cite{De02b}. We warn the reader that we are {\bf not} taking the $r$-associativity relation in \cite[\S3.3]{De02b}.

We would like to revisit the above notations by considering all $r\in\R$ altogether, following the framework in \cite[\S2.4]{Wa06}. Consider the (underlying) set $\cB^a(G,F):=\cB(G,F)\times\R$ which we will call the {\bf augmented building} for $G(F)$. An {\bf augmented apartment} is $\cA^a:=\cA\times\R\subset\cB^a(G,F)$ where $\cA\subset\cB(G,F)$ is an apartment in the usual sense.  We decompose $\cB^a(G,F)$ into {\bf augmented facets} in the following way: two elements $(x_1,r_1),(x_2,r_2)\in\cB^a(G,F)$ are in the same augmented facets iff $\fg(F)_{x_1\ge r_1}=\fg(F)_{x_2\ge r_2}$ and $\fg(F)_{x_1>r_1}=\fg(F)_{x_2>r_2}$. Denote by $\Phi^G$ the set of augmented facets. The group $G(F)$ acts on $\cB^a(G,F)$ by $g.(x,r):=(g.x,r)$. This action obviously acts on $\Phi^G$ and acts transitively on the set of augmented apartments. Observe that the set of $r$-facets can be identified $\Phi^G(r)=\{\sigma\in\Phi^G\;|\;\sigma\cap\left(\cB(G,F)\times\{r\}\right)\not=\emptyset\}$ as a subset of $\Phi^G$. We will use the two viewpoints on $\Phi^G(r)$ interchangeably.

For any augmented apartment $\cA^a\subset\cB^a(G,F)$, the collection
\[
\Phi^G_{\cA}:=\{\sigma\cap(\cA^a)\;|\;\sigma\subset\cB^a(G,F)\text{ an augmented facet, }\sigma\cap(\cA^a)\not=\emptyset\}
\]
is cut out by two types of hyperplanes: the first type is ``affine root hyperplanes'' which in the language of \cite[\S1.6, \S1.7]{Tits} are of the form $\{(x,r)\in\cA^a\;|\;\alpha(x)=r\}$ for some affine root $\alpha$. The second type is $\cA\times\{s\}$ for $s\in\Q$ such that $\fz_{\ge s}\supsetneq\fz_{>s}$ where $\fz$ is the Lie algebra of the torus $Z$ associated to $\cA$ as in \cite[\S2.2, \S3.2]{MP94}.  Note that the projection map $\cA^a\sra\cA$ sends any of the above hyperplane (of either type) isomorphically onto $\cA$.
\begin{definition}\label{def:critical} We call both types of the aforementioned hyperplanes on $\cA^a$ {\bf critical}, and call the open half-spaces cut out by a critical hyperplane a {\bf critical half-space}. For a critical hyperplane $P\subset\cA^a$, we say it is above\footnote{Since the real number is commonly called {\bf depth}, our convention is that $r_1$ is ``above'' $r_2$ iff $r_1<r_2$.} (resp. is below, passes through) a point $(x,r)\in\cA^a$ if $(x,r')\in P$ for some $r'<r$ (resp. $r'>r$, $r'=r$). We say a critical hyperplane is above (resp. is below, passes through) an augmented facet if it is above (resp. is below, passes through) any point in it.
\end{definition}

We have that $\Phi^G_{\cA}$ is the collection of possible intersections of a finite number of critical hyperplanes and critical half-spaces. This shows that every augmented facet is convex. For any $\sigma\in\Phi^G$, we denote by $\fg(F)_{\sigma}:=\fg(F)_{x=r}$ for any $(x,r)\in\sigma\subset\cB^a(G,F)\times\R$. Likewise $\fg(F)_{\ge\sigma}:=\fg(F)_{x\ge r}$ and $\fg(F)_{>\sigma}:=\fg(F)_{x>r}$.
Same for $\sigma\in\Phi^G_{\cA}$. From definition of Moy-Prasad filtration, we have
\begin{lemma}\label{lem:MPbasic} Let $\sigma_1,\sigma_2\in\Phi^G_{\cA}$. Then
	\begin{enumerate}
		\item $\fg(F)_{\ge\sigma_1}\subset\fg(F)_{\ge\sigma_2}$ iff every critical hyperplane on $\cA^a$ above $\sigma_2$ is also above $\sigma_1$.
		\item $\fg(F)_{>\sigma_1}\subset\fg(F)_{>\sigma_2}$ iff every critical hyperplane on $\cA^a$ below $\sigma_1$ is also below $\sigma_2$.
	\end{enumerate}
\end{lemma}
\begin{corollary}\label{cor:MPbasic} Suppose $\sigma_1,\sigma_2\in\Phi$ are such that $\sigma_1$ is contained in the closure of $\sigma_2$. Then
	\[
	\fg(F)_{>\sigma_1}\subset\fg(F)_{>\sigma_2}\subset\fg(F)_{\ge\sigma_2}\subset\fg(F)_{\ge\sigma_1}.
	\]
\end{corollary}

\begin{remark} In \cite[\S2.4]{Wa06}, Waldspurger considered what is essentially our $\Phi^G_{\cA}$ and denote it by $\Phi$. There is a natural injection $\Phi^G_{\cA}\subset\Phi^G$ so that $G(F).\Phi^G_{\cA}=\Phi^G$. As for most applications, working with the whole building versus a single apartment (in our notation, $\Phi^G$ versus $\Phi^G_{\cA}$) does not bring any essential difference, and we choose work with the seeming more general $\Phi^G$ merely because the additional cost is very small anyway. In fact, since the (extended) affine Weyl group acts transitively on the set of alcoves, we can furthermore restricts to those augmented facets that intersects a fixed alcove, and will do so in examples.
\end{remark}

For any $\sigma\in\Phi^G$, denote by $d(\sigma)$ the image of $\sigma\subset\cB(G,F)\times\R$ under the projection to $\R$, which is either a point or an open interval. Let us call $\sigma\in\Phi^G$ {\bf horizontal} of depth $r_0$ if $d(\sigma)=\{r_0\}$, and {\bf non-horizontal} of depth $(r_1,r_2)$ if $d(\sigma)=(r_1,r_2)$. We have the following contrapositive of \cite[Prop. 6.4]{MP94} and its corollary:

\begin{lemma}\label{lem:depthandnil} Suppose we have $x\in\cB(G,F)$, $s\in\R$ and $\gamma\in\fg(F)_{x\ge s}$ with $\dep(\gamma)>s$. Then the image of $\gamma$ in the quotient $\fg(F)_{x=s}$ is nilpotent. 
\end{lemma}

\begin{corollary}\label{cor:NH} Let $\sigma\subset\cB^a(G,F)$ be a non-horizontal augmented facet. Then every element in $\fg(F)_{\sigma}$ is nilpotent.
\end{corollary}

\begin{proof} Say $d_{\sigma}=(r_1,r_2)$. For every element $c\in \fg(F)_{\sigma}$, let $\gamma\in c+\fg(F)_{>\sigma}$ be an arbitrary lift. By definition of depth we have $\depth(\gamma)\ge r_2$. That $c$ is nilpotent follows from Lemma \ref{lem:depthandnil} with $s=(r_1+r_2)/2$ and some $x$ with $(x,s)\in\sigma$.
\end{proof}

Write $\dep(\sigma)=\sup d(\sigma)$; this is justified by that $\dep(\sigma)=\min_{\gamma\in\fg(F)_{\ge\sigma}}\dep(\gamma)$. Define a partial order on $\Phi^G$: for $\sigma_1,\sigma_2\in\Phi^G$ we say $\sigma_1\prec\sigma_2$ iff either $\dep(\sigma_1)<\dep(\sigma_2)$, or if $\dep(\sigma_1)=\dep(\sigma_2)$ and $\dim(\sigma_1)<\dim(\sigma_2)$. Define 
\begin{equation}\label{eq:defCg}
C_{G}:=\{(\sigma,c)\;|\;\sigma\in\Phi^G,\;c\in\fg(F)_{\sigma}\}.
\end{equation}
as well as
\[
	C_{G}^{nil}:=\{(\sigma,c)\in C_{G}\;|\;c\text{ is nilpotent}\}.
\]
For any $r\in\R$, the natural inclusion $\Phi(r)\subset\Phi$ induces $C_{G,r}\subset C_{G}$ and $C_{G,r}^{nil}\subset C_{G}^{nil}$. The goal of this section is 
\begin{theorem}\label{thm:graph} We will define a directed acyclic graph with vertex set $C_{G}$, depending on some choice made at each vertex, such that for any collection of choices we have
\begin{enumerate}
	\item There is a edge from $(\sigma_1,c_1)$ to $(\sigma_2,c_2)$ only if $c_1$ is nilpotent and $\sigma_1\prec\sigma_2$.
	\item For any $\gamma\in C_{G}$ and $r=\dep(\gamma)$, the union 
	\[	
	C_{G}(\gamma):=\bigcup_{s\in\R}c_{G,s}(\gamma)=\{(\sigma,c)\in C_{G}\;|\;\Ad(G(F))\gamma\cap\left(c+\fg(F)_{>\sigma}\right)\not=\emptyset\}\subset C_{G}\]
	is equal to the set of those vertices in $C_{G}$ which has a path to $c_{G,r}(\gamma)\subset C_{G,r}\bsl C_{G,r}^{nil}$.
\end{enumerate}	
\end{theorem}
Note that Corollary \ref{cor:NH} gives $C_{G}(\gamma)\subset C_{G,r}\cup C_{G}^{nil}$. By the nature of the statement, the theorem can and will be proved inductively for elements in $C_{G}^{nil}$, i.e. it suffices to show that for any vertex $(\sigma,c)\in C_{G,s}^{nil}$ of the graph, we have that $(\sigma,c)\in c_{G}(\gamma)$ iff it has an out-neighbor $(\sigma',c')$ (i.e. there is a directed edge from $(\sigma,c)$ to it) that is contained in $c_{G}(\gamma)$. These out-neighbors will be declared by two different methods (\ref{eq:defedge1}) and (\ref{eq:defedge2}) below. The two method will altogether cover all vertices of the graph, but they overlap in general. Moreover each method has the freedom given by some choice. That is, our graph will be defined up to some choices so that any choice gives a graph for which Theorem \ref{thm:graph} works, though some choices seem better when we analyze it in \S\ref{sec:geometric}.

Our first method of defining the out-neighbors, essentially in \cite[\S7.3]{Wa06}, applies to a non-horizontal augmented facet $\sigma$ with $d(\sigma)=(r_1,r_2)$. The intersection of $\cB(G,F)\times\{r_2\}$ with the closure of $\sigma$ is convex and equals to a union of horizontal augmented facets. They can be described as
\begin{definition}\label{def:below} Suppose $\sigma$ is a non-horizontal facet and $\sigma'\in\Phi^G$ is an augmented facet (necessarily horizontal) in the closure of $\sigma$ satisfying $d(\sigma')=\{\dep(\sigma)\}$. Then we call $\sigma'$ an augmented facet {\bf below} $\sigma$.
\end{definition}
\noindent {\it Choose} $\sigma'\in\Phi^G$ below $\sigma$.
By Lemma \ref{lem:MPbasic}, we have that
\begin{equation}\label{eq:pickedfacet}
\fg(F)_{\ge\sigma'}\supset\fg(F)_{\ge\sigma}\supset\fg(F)_{>\sigma}\supset\fg(F)_{>\sigma'}.
\end{equation}
We may realize $\fg(F)_{>\sigma}/\fg(F)_{>\sigma'}$ as a subspace of $\fg(F)_{\sigma'}$, and
for any $c\in\fg(F)_{\sigma}$ we have
\begin{equation}\label{eq:decomp}
c+\fg(F)_{>\sigma}=\bigsqcup_{c'\in c+\fg(F)_{>\sigma}/\fg(F)_{>\sigma'}}c'+\fg(F)_{>\sigma'}
\end{equation}
We define the set of out-neighbors of $(\sigma,c)$ to be the collection 
\begin{equation}\label{eq:defedge1}
\{(\sigma',c')\;|\;c'\in c+\fg(F)_{>\sigma}/\fg(F)_{>\sigma'}\}.
\end{equation}
Directly from (\ref{eq:decomp}) we see that $(\sigma,c)\in c_{G}(\gamma)$ iff at least one of $(\sigma',c')\in c_{G}(\gamma)$. Theorem \ref{thm:graph} is thus inductively true for this definition with any choice of $\sigma'$ made before (\ref{eq:pickedfacet}).

Our second method is a construction of Waldspurger in \cite[Lemma 7.2.2 \& Lemma 7.3.2]{Wa06} that applies to some $(\sigma,c)$ with $c$ nilpotent. We will review this construction in \S\ref{sec:geometric}. The point is that whenever $\sigma$ is horizontal, and sometimes also when $\sigma$ is not, Waldspurger produced from $(\sigma,c)$ another $(\sigma^{\dagger},c^{\dagger})$ where $\sigma^{\dagger}$ is a non-horizontal coset having $\sigma$ on its boundary and $\sigma\prec\sigma^{\dagger}$, such that any $\Ad(G(F))$-orbit in $\fg(F)$ meets $c+\fg(F)_{>\sigma}$ iff it meets $c^{\dagger}+\fg(F)_{>\sigma^{\dagger}}$; in fact in appropriate sense $c^{\dagger}=c$, see Lemma \ref{lem:>2}. We then declare, or rather construct the graph so that
\begin{equation}\label{eq:defedge2}
(\sigma,c)\text{ has a unique out-neighbor }(\sigma^{\dagger},c^{\dagger}).
\end{equation}
Theorem \ref{thm:graph} then follows.

\begin{remark}\label{rmk:conjugate} We can introduce a variant for the graph: that whenever we draw an edge from $(\sigma,c)$ to $(\sigma',c')$, we also draw an edge to $(\sigma,c)$ to $g.(\sigma',c')$ for any $g\in G(F)$. Equivalently, we can view our graph as defined on the vertex set $C_{G}/G(F)$. This obviously does not affect the validity of Theorem \ref{thm:graph}. Also, in (\ref{eq:defedge1}) instead of choosing one $\sigma'$ below $\sigma$ we can also choose all augmented facets below $\sigma$. That will not affect Theorem \ref{thm:graph} and provides a canonical (though not necessarily desirable) choice.
\end{remark}

\section{Analysis on Moy-Prasad quotients and nilpotent orbits}\label{sec:geometric}

From now on, by abuse of notation we will understand $C_{G}$ as a graph with directed edges constructed as in \S\ref{sec:graph}. The goal of this section is to analyze it, giving criteria for some $(\sigma,c)\in C_{G}$ to have or not have a directed edge to some other $(\sigma',c')$. We adapt all the notations in \S\ref{sec:graph}. Let us introduce a convenient notation from \cite{LY17}: suppose we have a homomorphism $\lambda:\mathbb{G}_m\ra\operatorname{End}(V)$ for a finite-dimensional vector space $V$ (over some working field), then for any $i\in\Z$ we write ${}^\lambda_iV\subset V$ for the eigenspace on which $\lambda(t).v=t^iv$ for $t\in\mathbb{G}_m$. We also write ${}^\lambda_{\ge j}V:=\bigoplus_{i\ge j}{}^{\lambda}_iV$, etc..

\subsection{Moy-Prasad quotients \`{a} la Reeder-Yu}\label{subsec:MPQ}

We will need a theory for Moy-Prasad quotients ``in the tame case,''  for which we review the results in \cite[\S4]{RY14} with our language.
Denote by $\Z_{(p)}$ the localization of $\Z$ at $p$, namely the ring of rational numbers with denominators coprime to $p$. The subset $\cB(G,F)_{(p)}\subset\cB(G,F)$ of those points in $\cB(G,F)$ with rational coordinates with denominator coprime to $p$ is well-defined, e.g. they are those points that become hyperspecial vertices after a tamely ramified base change. When $\cA$ corresponds to a maximal $F$-split torus $S$, we have $\cA_{(p)}:=\cA\cap\cB(G,F)_{(p)}$ is a torsor of $X_*(S)\otimes_{\Z}\Z_{(p)}$.  

Denote by $F^s/F$ a fixed separable closure and $F^{ur}/F$ the maximal unramified extension in it, with residue field $\bar{k}$ an algebraic closure of $k$. Let $n'\in\Z_{>0}$ be coprime to $p$, $r=\frac{a'}{m'}$ with $a'\in\Z$ (we will replace them by $a$ and $m$ in a moment), and $x\in\cB(G,F)_{(p)}$ be such that it is a hyperspecial vertex in $\cB(G,E_{m'})$ where $E_{m'}/F^{ur}$ is the unique (up to isomorphism) tame extension with ramification index $m'$. Fix a uniformizer $\varpi_F\in F$. For any $b_1\equiv b_2(\operatorname{mod}m')$, we have $\fg(F)_{x=\frac{b_1}{m'}}\cong\fg(F)_{x=\frac{b_2}{m'}}$ by multiplying an appropriate power of $\varpi_F$. Using such identification, the direct sum
\[
\fg^{\clubsuit}_x:=\bigoplus_{b\in\Z/m'}\fg(F)_{x=\frac{b}{m'}}
\]
is a Lie algebra over $k$. Over $\bar{k}$, Reeder-Yu showed \cite[Thm. 4.1]{RY14} that it can be identified with the reductive quotient (Lie algebra version) of the hyperspecial parahoric of $G/_{E_{m'}}$ at $x$. They also consider the reductive quotient (group version) $\sG_x^{\clubsuit}$ for which they give a Frobenius action so that it becomes a connected reductive group over $k$ with Lie algebra isomorphic to $\fg^{\clubsuit}_x$. We have in particular that $\sG_x^{\clubsuit}/_k$ and $G/_F$ has the same absolute root data. Moreover, there is a $k$-homomorphism $\rho':\mu_{m'}\ra\operatorname{Aut}(\sG_x^{\clubsuit})$ such that $\fg(F)_{x=\frac{b}{m'}}$ is identified with the eigenspace on which $\rho'(z)v=z^bv$ for $z\in\mu_{m'}$, $v\in\fg^{\clubsuit}_x$, $b\in\Z/m'$. We denote ${}^{\rho'}_b\fg^{\clubsuit}_x$ such eigenspace, and denote by ${}^{\rho'}_0 \sG_x^{\clubsuit}$ the connected subgroup invariant by $\rho'$ with $\Lie({}^{\rho'}_0 \sG_x^{\clubsuit})={}^{\rho'}_0\fg^{\clubsuit}_x$. Analogous to ${}^{\rho'}_0\fg_x\cong\fg(F)_{x=0}$ we have ${}^{\rho'}_0\sG_x^{\clubsuit}\cong\sG_x$, the reductive quotient of $G$ over $F$ at $x$ as an algebraic group over $k$. The following lemma follows from \cite[Lemma 4.1.2]{De02b}:

\begin{lemma}\label{lem:nilpotent} For $b\in\Z$, an element in $\fg(F)_{x=\frac{b}{m'}}$ is nilpotent in the sense of Definition \ref{def:nil} iff it is nilpotent as an element in the reductive Lie algebra $\fg^{\clubsuit}_x$.
\end{lemma}

Write $m=m'/\gcd(a',m')$, $a=a'/\gcd(a',m')$. Let
\[
\fg_x^{\heartsuit}:=\bigoplus_{b\in (\gcd(a',m')\Z)/m'\Z}\fg(F)_{x=\frac{b}{m'}}.
\]
With
\[
1\ra\mu_{\gcd(a',m')}\ra\mu_a'\ra\mu_{a}\ra 1
\]
we have 
$\fg_x^{\heartsuit}=(\fg_x^{\clubsuit})^{\rho'(\mu_{\gcd(a',m')})}$. Let $\sG_x^{\heartsuit}:=\left((\sG_x^{\clubsuit})^{\rho'(\mu_{\gcd(a',m')})}\right)^o$ so that we have an induced $\rho:\mu_{m}\ra\operatorname{Aut}(\sG_x^{\heartsuit})$. For any $b\in\Z$ and $b'=\gcd(a',m')\cdot b$, this gives $\fg(F)_{x=\frac{b}{m}}=\fg(F)_{x=\frac{b'}{m'}}={}^{\rho'}_{b'}\fg_x^{\clubsuit}={}^{\rho}_{b}\fg_x^{\heartsuit}$. In particular this holds for $b=a$ and $b=-a$, so that e.g. $\fg(F)_{x=r}={}^{\rho}_{a}\fg_x^{\heartsuit}$. We remark that $m=m'/\gcd(a',m')$ is the denominator of $r$ in its reduced expression.

Suppose we are given $c\in\fg(F)_{x=r}$ that is nilpotent. We may complete $c$ into an $\slt$-triple
\begin{equation}\label{eq:gradedtriple}
(c,h,d)\in\fg(F)_{x=r}\times\fg(F)_{x=0}\times\fg(F)_{x=-r}\cong{}^\rho_{a}\fg^{\heartsuit}_x
\times{}^\rho_0\fg^{\heartsuit}_x\times{}^\rho_{-a}\fg^{\heartsuit}_x.
\end{equation}
Viewing the triple in $\fg^{\heartsuit}_x$, we can lift it to $\varphi:SL_2/_k\ra \sG_x^{\heartsuit}$ (see e.g. \cite[\S5.5]{Ca93}). Restricting to the diagonal torus $\mathbb{G}_m\subset SL_2$ we get a cocharacter $\lambda:\mathbb{G}_m\ra \sG_x^{\heartsuit}$.
Note that $\varphi$ and thus $\lambda$ are uniquely determined by the triple, and the triple is almost invariant under $\rho:\mu_m\ra\operatorname{Aut}(\sG_x^{\heartsuit})$; more precisely, every $\rho(z)$ scales $c$ by $z^a$ and $d$ by $z^{-a}$, which has no effect on the diagonal. This implies that $\lambda$ has image in ${}^{\rho}_0\sG_x^{\heartsuit}$. Consequently, we have
\[
\fg^{\heartsuit}_x=\bigoplus_{b\in\Z/m}\bigoplus_{i\in\Z}{}^{\rho}_b({}^{\lambda}_i\fg^{\heartsuit}_x).
\]
That is, the two gradings given by $\rho$ and $\lambda$ are compatible, and we may realize
\begin{equation}\label{eq:bigrade}
(c,h,d)\in{}^\lambda_2\fg(F)_{x=r}\times{}^\lambda_0\fg(F)_{x=0}\times{}^\lambda_{-2}\fg(F)_{x=-r}
={}^\rho_a({}^\lambda_2\fg^{\heartsuit}_x)\times{}^\rho_0({}^\lambda_0\fg^{\heartsuit}_x)\times{}^\rho_{-a}({}^\lambda_{-2}\fg^{\heartsuit}_x).
\end{equation}
Recall that the image of $\lambda$ lies in ${}^\rho_0\sG_x^{\heartsuit}$ which we identify as the reductive quotient $\sG_x$. Thus there is maximal $k$-split torus $\sS\subset\sG_x$ that contains the image of $\lambda$. We may choose a maximal $F$-split torus $S\subset G$ whose corresponding apartment $\cA$ contains $x$ and whose reductive quotient is identified with $\sS$. Under the identification $X_*(\sS)=X_*(S)$ we have that $\lambda$ is identified with $\til{\lambda}:\mathbb{G}_m/_F\ra S$. By \cite[Cor. 4.3.2]{De02b} we can lift $(c,h,d)$ from (\ref{eq:gradedtriple}) and (\ref{eq:bigrade}) to
\begin{equation}\label{eq:liftedsl2}
(\til{c},\til{h},\til{d})\in {}^{\til{\lambda}}_2\fg(F)_{x\ge r}\times{}^{\til{\lambda}}_0\fg(F)_{x\ge 0}\times{}^{\til{\lambda}}_{-2}\fg(F)_{x\ge -r}
\end{equation}

As before let us denote by $\cO(x,r,c)$ the minimal nilpotent $\Ad(G(F))$-orbit that intersects $c+\fg(F)_{x>r}$. By \cite[Cor. 5.2.4]{De02b}, we have that $\cO(x,r,c)=\Ad(G(F))\til{c}$. Since the weighted Dynkin diagram associated to $c$ and $\til{c}$ are determined by $\lambda$ and $\til{\lambda}$ respectively. Hence the geometric orbits $\Ad(\sG_x^{\clubsuit}(\bar{k}))c$ and $\Ad(G(F^s))\til{c}$, although defined over different fields, are ``the same'' as they are classified by the same weighted Dynkin diagram; we refer to \cite{CM93} for the notion of weighted Dynkin diagrams. For example in classical groups they correspond to the same partition. Note that we have to use $\sG_x^{\clubsuit}$ instead of $\sG_x^{\heartsuit}$ cause only the former has the same absolute root data as $G$. We highlight this result:

\begin{proposition}\label{prop:samelift} Let $\til{c}=\cO(x,r,c)$. The geometric orbits $\Ad(\fg_x^{\clubsuit}(\bar{k}))c$ and $\Ad(G(F^s))\til{c}$ are ``the same'' in the sense that they have the same weighted Dynkin diagram.
\end{proposition}

For any finite extension $E/F$, denote by $\cO^E(x,r,c)$ the minimal $\Ad(G(E))$-orbit that intersects $c+\fg(E)_{x>r}$. 

\begin{corollary}\label{cor:extension} $\Ad(G(E))\cO(x,r,c)=\cO^E(x,r,c)$ for any finite $E/F$.
\end{corollary}

\begin{proof} By definition the $\Ad(G(E))\cO(x,r,c)\ge\cO^E(x,r,c)$. However Proposition \ref{prop:samelift} says they live in the same geometric orbits. Hence they are equal.
\end{proof}

\noindent With the corollary we may also denote by $\cO^{ur}(x,r,c)$ (resp. $\cO^s(x,r,c)$) the limit of $\cO^E(x,r,c)$ for $E/F$ unramified (resp. any $E/F$) and big enough, which necessarily stabilizes and is the minimal $\Ad(G(F^{ur}))$-orbit (resp. $\Ad(G(F^s))$-orbit) which meets $c+\fg(F^{ur})_{x>r}$ (resp. $c+\fg(F^s)_{x>r}$).

Let $S\subset G$ be a maximal $F$-split torus with corresponding apartment $\cA$ such that $x\in\cA$. Let $\til{\lambda}:\Gm/_F\ra S$ be any cocharacter. As before we identify the reductive quotient $\sS$ of $S$ as a maximal $k$-split torus of $\sG_x$, and we denote by $\lambda:\Gm/_k\ra\sS\subset\sG_x$ the corresponding cocharacter. For any $v\in X_*(S)\otimes_{\Z}\Z_{(p)}$ we denote by $x+v\in\cA_{(p)}$ the action of $v$ on $x$. Unfolding the definition of Moy-Prasad filtration, we have for any`$\ell\in\Z$ and $s\in\R$ that
\begin{equation}\label{eq:MPequal}
{}^{\lambda}_{\ell}\fg(F)_{x=r}={}^{\lambda}_{\ell}\fg(F)_{(x+s\til{\lambda})=r+\ell s},\;\;{}^{\til{\lambda}}_{\ell}\fg(F)_{x\ge r}={}^{\til{\lambda}}_{\ell}\fg(F)_{(x+s\til{\lambda})\ge r+\ell s},
\end{equation}

\begin{lemma}\label{lem:O} For any $t\in\R$ we have $\cO(x,r,c)=\cO(x+t\til{\lambda},r+\ell t,c)$.
\end{lemma}

\begin{proof} Observe that $\cO(x+s\til{\lambda},r+\ell s,c)$ depends only on the augmented facet that $(x+s\til{\lambda},r+\ell s)$ lives in. By standard continuity argument it suffices to prove that there exists $\delta>0$ such that for all $s\in (-\delta,\delta)\cap\Z_{(p)}$, we have $\cO(x,r,c)=\cO(x+s\til{\lambda},r+\ell s,c)$. Take $\delta$ small enough so that for any $s\in (-\delta,\delta)$, the augmented facet containing $(x+s\til{\lambda},r+\ell s)$ has its closure contains $(x,r)$. This implies that $c+\fg(F)_{x>r}\subset c+\fg(F)_{x+s\til{\lambda}> r+\ell s}$ and therefore $\cO(x,r,c)\le\cO(x+s\til{\lambda},r+\ell s,c)$.
	
Proposition \ref{prop:samelift} asserts that $\cO(x,r,c)$ and $\cO(x+s\til{\lambda},r+\ell s,c)$ have the same weighted Dynkin diagram as that of $c$ in the corresponding $\fg_x^{\clubsuit}$ and $\fg_y^{\clubsuit}$ where $y=x+s\til{\lambda}$. Suppose $s=\frac{b}{n}$ with $b,n\in\Z$ and $n$ coprime to $p$. Then $\til{\lambda}(\varpi^{b/n})$ sends $\fg_x^{\clubsuit}$ to $\fg_y^{\clubsuit}$ and maps $c$ to $c$. This shows that the weighted Dynkin diagram of $c\in \fg_x^{\clubsuit}$ and $c\in\fg_y^{\clubsuit}$ can be identified by an element in $G(F^s)$, and consequently $\cO(x,r,c)$ and $\cO(x+s\til{\lambda},r+\ell s,c)$ have the same weighted Dynkin diagram, i.e. they are in the same $\Ad(G(F^s))$-orbit. This together with $\cO(x,r,c)\le\cO(x+s\til{\lambda},r+\ell s,c)$ shows that they are the same $\Ad(G(F))$-orbit.
\end{proof}

Suppose we have another nilpotent element $c'\in c+{}^\rho_a({}^{\lambda}_{<\ell}\fg_x^{\heartsuit})$. We end this subsection with another lemma:

\begin{lemma}\label{lem:comparelift} We have that $\cO(x,r,c)\le\cO(x,r,c')$. If there is an equality, then $c'\in\Ad({}^{\rho}_0\sG_x^{\heartsuit}(k))c$.
\end{lemma}

\begin{proof} By definition of Moy-Prasad filtration, for $s\in\R_{>0}$ small enough we have 
\[\fg(F)_{x+s\til{\lambda}>r+\ell s}=\fg(F)_{x=r}\times_{{}^{\rho}_a\fg_x^{\heartsuit}}{}^{\rho}_a({}^{\lambda}_{<\ell}\fg_x^{\heartsuit})\]
and thus
\[c'+\fg(F)_{x>r}\subset c+\fg(F)_{x+s\til{\lambda}>r+\ell s}.\]
By definition of $\cO(x,r,c')$, this gives $\cO(x,r,c')\ge\cO(x+s\til{\lambda},r+\ell s,c)=\cO(x,r,c)$ where the last equality follows from Lemma \ref{lem:O}. Now suppose $\cO(x,r,c')=\cO(x,r,c)$. By Proposition \ref{prop:samelift}, this shows that $c,c'\in\fg_x^{\clubsuit}$ have the same weighted Dynkin diagram, i.e. $\Ad(\sG_x^{\clubsuit}(\bar{k}))c=\Ad(\sG_x^{\clubsuit}(\bar{k}))c'$. 
By Lemma \ref{lem:test}, this implies $c'\in\Ad({}^{\rho}_0\sG_x^{\heartsuit}(k))c$ as desired.
\end{proof}

\subsection{Edges as in (\ref{eq:defedge2})}\label{subsec:edge2}

We continue the setup in the previous subsection and study $C_{G}$. Recall that one of the two constructions of the edges, namely (\ref{eq:defedge2}), refers to the proof of \cite[Lemma 7.3.2]{Wa06}. We review and study it below. Suppose we are given $(\sigma,c)\in C_{G}^{nil}$. \cite[Lemma 2.4.1]{Wa06} asserts that $\sigma\cap(\cB(G,F)_{(p)}\times\Z_{(p)})$ is dense in $\sigma$. Hence we may pick $(x,r)\in \cB(G,F)_{(p)}\times\Z_{(p)}$ with $r=\frac{m}{n}$ is the reduced expression as in the previous subsection. We had (after some choices) tori $\sS\subset\sG_x$, $S\subset G$ and its corresponding apartment $\cA$, and $\til{\lambda}\in X_*(S)$.

Let $\cA^a=\cA\times\R$ be the augmented apartment, which is a torsor of $(X_*(S)\otimes_{\Z}\R)\times\R$. The idea is to ``walk'' towards the direction of $(\til{\lambda},2)\in(X_*(S)\otimes_{\Z}\R)\times\R$ from $(x,r)\in\cA^a$. There are two possible scenarios:
\begin{enumerate}[label=(\Alph*)]
	\item\label{enumi:A} There exists $\delta>0$ such that for any $0\le s<\delta$, $(x,r)+s(\til{\lambda},2)=(x+s\til{\lambda},r+2s)\in\sigma$, or
	\item\label{enumi:B} There exists a non-horizontal augmented facet $\sigma^{\dagger}$ for which
	\begin{enumerate}[label=(\roman*)]
		\item $\sigma$ is contained in the boundary of $\sigma^{\dagger}$,
		\item $\sigma\prec\sigma^{\dagger}$ by the partial order defined before (\ref{eq:defCg}), and
		\item There exists $\delta>0$ such that for any $0<s<\delta$, $(x,r)+s(\til{\lambda},2)=(x+s\til{\lambda},r+2s)\in\sigma^{\dagger}$.
	\end{enumerate} 
\end{enumerate}
Apparently, if $\sigma$ is horizontal we are in scenario \ref{enumi:B}. The construction of the edges in (\ref{eq:defedge2}) can be applied whenever we are in scenario \ref{enumi:B}, which we assume so in the rest of this subsection. Using the definition of Moy-Prasad quotient, just like in (\ref{eq:MPequal}) we have 
\begin{equation}\label{eq:MP>2}
\fg(F)_{\ge\sigma^{\dagger}}=\fg(F)_{\ge\sigma}\times_{\fg(F)_{\sigma}}{}^{\lambda}_{\ge 2}\fg(F)_{\sigma},\;\;\fg(F)_{>\sigma^{\dagger}}=\fg(F)_{\ge\sigma}\times_{\fg(F)_{\sigma}}{}^{\lambda}_{>2}\fg(F)_{\sigma},\;\;\fg(F)_{\sigma^{\dagger}}={}^{\lambda}_{2}\fg(F)_{\sigma}
\end{equation}
We have
\begin{lemma}\label{lem:>2}
	Every element in $c+{}^{\lambda}_{>2}\fg(F)_{\sigma}+\fg(F)_{>\sigma}=c+\fg(F)_{>\sigma^{\dagger}}$ is conjugate under $G(F)$ to an element in $c+\fg(F)_{>\sigma}$.
\end{lemma}
\begin{proof} We sketch essentially the same proof in \cite[Lemma 7.2.2 \& 7.3.2]{Wa06}. Note that $c\in \fg(F)_{\sigma}={}^{\rho}_a\fg_x^{\heartsuit}$ and ${}^{\lambda}_{>2}\fg(F)_{\sigma}={}^{\rho}_a({}^\lambda_{>2}\fg_x^{\heartsuit})$. By the property of $\slt$-triple as in (\ref{eq:bigrade}), we have ${}^{\rho}_a({}^\lambda_{>2}\fg_x^{\heartsuit})=[c,{}^{\rho}_0({}^\lambda_{>0}\fg_x^{\heartsuit})]$. From this one deduces that every element in $c+{}^{\lambda}_{>2}\fg(F)_{\sigma}$ is conjugate under $\exp({}^\lambda_{>0}\fg(F)_{x=0})\subset\sG_x(k)$ to $c$. Consequently, every element in $c+{}^{\lambda}_{>2}\fg(F)_{\sigma}+\fg(F)_{>\sigma}$ is conjugate under the parahoric $G(F)_{x\ge 0}$ to an element in $c+\fg(F)_{>\sigma}$.
\end{proof}

The lemma is exactly what we claimed and used for Theorem \ref{thm:graph}. That is, in (\ref{eq:defedge2}) we put $c^{\dagger}\in {}^{\lambda}_2\fg(F)_{\sigma^{\dagger}}$ to be the element identified with $c$. We highlight that Lemma \ref{lem:O} gives 
\begin{equation}\label{eq:samelift}
\cO(\sigma,c)=\cO(\sigma^{\dagger},c^{\dagger}).
\end{equation}

\begin{remark}\label{rmk:conjugate2} One can show that once $S$ and $\cA$ are fixed, $\sigma^{\dagger}$ does not depend on the choice of $(x,r)\in\sigma\cap\cA^a_{(p)}$. Based on this, one can furthermore show that up to $(\sigma^{\dagger},c^{\dagger})$ depends only on $(\sigma,c)$ up to the action of the stabilizer of $(\sigma,c)$ in $G(F)$. As in Remark \ref{rmk:conjugate}, one can declare edges from $(\sigma,c)$ to all such $(\sigma^{\dagger},c^{\dagger})$, so that we don't need to make any choice at $(\sigma,c)$.
\end{remark}

\subsection{Edges as in (\ref{eq:defedge1})}\label{subsec:edge1}
We turn to the edges defined in (\ref{eq:defedge1}) which are actually the crucial ones despite the seemingly innocent construction. We are given a non-horizontal augmented facet $\sigma$ and another horizontal augmented facet $\sigma'$ below it, following the choice made after Definition \ref{def:below}.
From now on, let $(x,r)\in \sigma'\cap\cA^a_{(p)}$ be any choice. Here $\cA^a_{(p)}:=\cA_{(p)}\times\Z_{(p)}$ is any augmented apartment intersecting $\sigma$ and thus $\sigma'$. Say $S\subset G$ is the corresponding maximal $F$-split torus. There exists $\til{\lambda}\in X_*(S)$ and $\ell\in\Z_{>0}$ such that there exists $\delta>0$ such that $0<s<\delta\implies (x,r)+(s\til{\lambda},-\ell s)\in\sigma$. We have $\sS\subset\sG_x$ the reductive quotient of $S$, with $X_*(S)=X_*(\sS)$ so that we denote by $\lambda\in X_*(\sS)$ the image of $\til{\lambda}$. Unfolding the definition of Moy-Prasad filtration we have
\begin{equation}\label{eq:MP<2}
	\fg(F)_{\ge\sigma}=\fg(F)_{\ge\sigma'}\times_{\fg(F)_{\sigma'}}{}^{\lambda}_{\le \ell}\fg(F)_{\sigma'},\;\;\fg(F)_{>\sigma}=\fg(F)_{\ge\sigma'}\times_{\fg(F)_{\sigma'}}{}^{\lambda}_{<\ell}\fg(F)_{\sigma'},\;\;\fg(F)_{\sigma}={}^{\lambda}_{\ell}\fg(F)_{\sigma'}.
\end{equation}

By (\ref{eq:defedge1}), we draw an edge from $(\sigma,c)$ to $(\sigma',c')$ for any $c'\in c+\fg(F)_{>\sigma}/\fg(F)_{>\sigma'}$, i.e. for any $c'\in c+{}^{\lambda}_{<\ell}\fg(F)_{\sigma'}$. Let $m\in\Z_{>0}$ be the denominator of $r=\frac{a}{m}$ as in \S\ref{subsec:MPQ}, we have that $c\in{}^\lambda_{\ell}\fg(F)_{\sigma'}={}^{\rho}_a({}^{\lambda}_{\ell}\fg_x^{\heartsuit})$ and ${}^{\lambda}_{<\ell}\fg(F)_{\sigma'}={}^{\rho}_a({}^{\lambda}_{<\ell}\fg_x^{\heartsuit})$. Recall also that ${}^{\rho}_0\sG_x^{\heartsuit}=\sG_x$. As in Remark \ref{rmk:conjugate2} we are happy to conjugate our elements by $\sG_x$ whenever applicable. This brings us to the geometric problem, which we highlight below:
\begin{problem}\label{prob:hard} Given
\begin{enumerate}
	\item A connected reductive group $\sG_x^{\heartsuit}$ over our finite field $k$ with $\fg_x^{\heartsuit}:=\Lie \sG_x^{\heartsuit}$,
	\item An integer $m>0$ prime to $p$ and an algebraic group homomorphism $\rho:\mu_m\ra\operatorname{Aut}(\sG_x^{\heartsuit})$,
	\item A cocharcter $\lambda:\mathbb{G}_m\ra{}^\rho_0\sG_x^{\heartsuit}$, a residue class $a\in\Z/m$, an integer $\ell>0$, and an element $c\in {}^\rho_a({}^\lambda_{\ell}\fg^{\heartsuit}_x)$.
\end{enumerate}	
What are the $\Ad({}^{\rho}_0\sG_x^{\heartsuit}(k))$-orbits in ${}^\rho_a\fg_x^{\heartsuit}$ that intersects $c+{}^\rho_a({}^\lambda_{<\ell}\fg^{\heartsuit}_x)$?
\end{problem}

Consider any $c'\in c+{}^\rho_a({}^\lambda_{<\ell}\fg^{\heartsuit}_x)$. We begin by the simpler situation when $c'$ is nilpotent. Recall that $\cO(\sigma',c')$ is the minimal $\Ad(G(F))$-orbit that intersects $c'+\fg(F)_{>\sigma'}=c'+\fg(F)_{x>r}$. We have

\begin{corollary}\label{cor:edgecond2} Suppose $c'\in c+{}^\rho_a({}^\lambda_{<\ell}\fg^{\heartsuit}_x)$ is nilpotent. Then $\cO(\sigma',c')\ge\cO(\sigma',c)=\cO(\sigma,c)$. When the equality holds, we have $c'\in\Ad({}^\rho_0\sG_x^{\heartsuit}(k))c=\Ad(\sG_x(k))c$.
\end{corollary}

\begin{proof} The equality $\cO(\sigma',c)=\cO(\sigma,c)$ comes from Lemma \ref{lem:O} while the rest of the corollary come from Lemma \ref{lem:comparelift}.
\end{proof}

\begin{corollary}\label{cor:trivexist} Under the identification $\fg(F)_{\sigma}={}^{\lambda}_{\ell}\fg(F)_{\sigma'}$ there is an edge from $(\sigma,c)$ to $(\sigma',c)$, with $\cO(\sigma,c)=\cO(\sigma',c)$. Any edge from $(\sigma,c)$ to $(\sigma',c')$ such that $c'$ is nilpotent and $\cO(\sigma,c)=\cO(\sigma',c')$ is $G(F)$-conjugate to the edge in the first sentence.
\end{corollary}

\begin{proof} That there is an edge from $(\sigma,c)$ to $(\sigma',c)$ follows immediately from definition. The rest of the statements follows from Corollary \ref{cor:edgecond2}. 
\end{proof}

Now we turn the case of non-nilpotent $c'\in c+{}^\rho_a({}^\lambda_{<\ell}\fg^{\heartsuit}_x)$, for which we can get weaker results. Firstly we have a minor generalization of \cite[Prop. 3.6]{CO23} (in {\it loc. cit.} it is assumed that $\lambda$ is associated to an $\slt$-triple completing $c$), which is just a rephrasement of Lemma \ref{lem:test}:

\begin{lemma}\label{lem:CO3.6} The element $c$ is contained in the Zariski closure of $\bar{k}\cdot\Ad(({}^\rho_0\sG_x^{\heartsuit})(\bar{k}))c'$.
\end{lemma}

Because the dimension of centralizer for an element in a Lie algebra is lower semi-continuous (it is the rank of the adjoint action), the lemma implies

\begin{corollary}\label{cor:edgecond1} We have $\dim\Ad(\sG_x^{\heartsuit}(\bar{k}))c'\ge \dim \Ad(\sG_x^{\heartsuit}(\bar{k}))c$.
\end{corollary}

Let $c'=c_s'+c_n'$ be the Jordan decomposition. Consider the nilpotent geometric $\Ad(\sG_x^{\heartsuit})(\bar{k})$-orbit $N(c')\subset\fg_x^{\heartsuit}$ given by the Lusztig-Spaltenstein induction of the orbit $\Ad(\sG_x^{\heartsuit}(\bar{k}))c_n'$ from $Z_{\sG_x^{\heartsuit}}(c_s')$ to $\sG_x^{\heartsuit}$; see Definition \ref{def:LS}. Ciubotaru and Okada observed \cite[Lemma 4.2]{CO23} that the closure of $N(c')$ is the intersection of $\bar{k}\cdot\Ad(({}^\rho_0\sG_x^{\heartsuit})(\bar{k}))c'$ with the nilpotent cone of $\fg_x^{\heartsuit}$. With Lemma \ref{lem:CO3.6} we have

\begin{corollary}\label{cor:CO4.2} (\cite[Lemma 4.2]{CO23}) The element $c$ is contained in the Zariski closure of $N(c')$. 
\end{corollary}

Note that $N(c')$ is only a $\sG_x^{\heartsuit}(\bar{k})$-orbit, not a $({}^{\rho}_0\sG_x^{\heartsuit})(\bar{k})$-orbit. In general $N(c')\cap{}^{\rho}_a(\fg_x^{\heartsuit})(\bar{k})$ is the closure of the union of a number of $\sG_x^{\heartsuit}(\bar{k})$-orbits, say $\{N^{\rho}_1(c'),...,N^{\rho}_n(c')\}$. Recall that $(x,r)\in\sigma'$. Similar to before we denote by $\cO^{ur}(\sigma',N^{\rho}_i(c'))$ the minimal nilpotent $\Ad(G(F^{ur}))$-orbit that meets $N^{\rho}_i(c')+\fg(F^{ur})_{>\sigma'}$. We have the following corollary to Corollary \ref{cor:CO4.2}, which is implicit in \cite[\S5]{CO23}:

\begin{corollary}\label{cor:CO5} We have $\cO(\sigma,c)=\cO(\sigma',c)$ is contained in the $p$-adic closure of one of $\cO^{ur}(\sigma',N^{\rho}_i(c'))$ ($1\le i\le n$) above.
\end{corollary}

\begin{proof} By definition $c$ is contained in the Zariski closure of one of $N^{\rho}_i(c')$. Thanks to Corollary \ref{cor:extension}, we have $\Ad(G(F^{ur}))\cO(\sigma,c)=\cO^{ur}(\sigma,c)$. The latter is contained in the $p$-adic closure of $\cO^{ur}(\sigma',N^{\rho}_i(c'))$ thanks to Lemma \ref{lem:comparelift} applied with $F$ replaced by $F^{ur}$.
\end{proof}

Corollary \ref{cor:edgecond2} to Corollary \ref{cor:CO5} all give necessary conditions for the existence of an edge in $C_{G}$ from $(\sigma,c)$ to $(\sigma',c')$, with the exception of Corollary \ref{cor:trivexist} which also contains a rather naive existence result. When $r\in\Z$, we have $m=1$ so that $\rho$ is trivial. In this case we will prove a stronger existence result. To accurately state what we can achieve we need more notions. Define an equivalence relation $\sim_{semi}$ on $\fg^{nil}(F)/\!\!\sim$ as follows:

\begin{definition}\label{def:semi} We denote by $\sim_{semi}$ the equivalence relation on $\fg^{nil}(F)/\!\!\sim$ generated by the following relation: For $\cO_1,\cO_2\in\fg^{nil}(F)/\!\!\sim$ we declare $\cO_1\sim_{semi}\cO_2$ if there exist $r\in\Z$, $x\in\cB(G,F)$ and $c_1,c_2\in\fg(F)_{x=r}$ such that $\cO_i=\cO(x,r,c_i)$ and $\Ad(\sG_x(\bar{k}))c_1=\Ad(\sG_x(\bar{k}))c_2$.
\end{definition}

The following lemma will not be used in this article; it is just there so that Definition \ref{def:semi} looks slightly simpler.

\begin{lemma} In Definition \ref{def:semi}, it suffices to take $r=0$, i.e. those relations given by $r=0$ generate the same equivalence relation.
\end{lemma}

\begin{proof} Suppose $\cO_1,\cO_2\in\fg^{nil}(F)/\!\!\sim$, $r\in\Z$, $x\in\cB(G,F)$ and $c_1,c_2\in\fg(F)_{x=r}$ are such that $\cO_i=\cO(x,r,c_i)$ and $\Ad(\sG_x(\bar{k}))c_1=\Ad(\sG_x(\bar{k}))c_2$. Then $c_1,c_2$ has the same weighted Dynkin diagram in $\fg(F)_{x=r}\cong\fg(F)_{x=0}\cong\Lie\sG_x$. Hence their $\slt$-triples have the same associated cocharacter. In particular there exists $\lambda:\Gm/_k\ra\sS\subset\sG_x$ such that $c_1,c_2\in{}^{\lambda}_2\fg(F)_{x=r}$, and that $\Ad({}^{\lambda}_0\sG_x(\bar{k}))c_1=\Ad({}^{\lambda}_0\sG_x(\bar{k}))c_2$. Identify $\lambda$ as $\til{\lambda}\in X_*(S)$ where $S\subset G$ is a maximal $F$-split torus with reductive quotient $\sS$ and such that the corresponding apartment contains $x$. Then Lemma \ref{lem:O} gives $\cO(x,r,c_i)=\cO(x-\frac{r}{2}\til{\lambda},0,c_i)$. Write $y=x-\frac{r}{2}\til{\lambda}$. We have ${}^{\lambda}_0\sG_y={}^{\lambda}_0\sG_x$. Thus $\Ad(\sG_y(\bar{k}))c_1=\Ad(\sG_y(\bar{k}))c_2$ and by replacing $x$ with $y$ we have reduced to the case $r=0$.
\end{proof}

\begin{definition}\label{def:semiunram} We denote by $\fg^{nil}(F)/\!\!\sim_{semi}$ the quotient set of $\fg^{nil}(F)/\!\!\sim$ by $\sim_{semi}$. We give $\fg^{nil}(F)/\!\!\sim_{semi}$ the partial order generated by that in $\fg^{nil}$, i.e. for {\it any} $\cO_1\le\cO_2$ we declare $\cO_1/\!\sim_{semi}\,\le \cO_2/\!\sim_{semi}$ and consider the partial order relation this generates.
	
For $(\sigma,c)\in C_G$ or $(x,r)\in\cB^a(G,F)$, $c\in\fg(F)_{x=r}$ we denote by $\cO^{semi}(\sigma,c)$ (resp. $\cO^{semi}(x,r,c)$) the image of $\cO(\sigma,c)$ (resp. $\cO(x,r,c)$) under the quotient map. For $\gamma\in\fg(F)$ we also denote by $\WF^{semi}(\gamma)\subset\fg^{nil}(F)/\!\!\sim_{semi}$ the set of maximal elements in $\fg^{nil}(F)/\!\!\sim_{semi}$ among the image of $\WF(\gamma)$ under the quotient map. We call $\WF^{semi}(\gamma)$ the {\bf semi-unramified wave-front set}.
\end{definition}

Denote by $\fg^{nil}(F^{ur})/\!\!\sim$ the set of $\Ad(G(F^{ur}))$-orbits in $\fg^{nil}(F^{ur})$, with the partial order given by $p$-adic closure. We denote by $\WF^{ur}(\gamma)$ the set of maximal elements in the image of $\WF(\gamma)$ under $\fg^{nil}(F)/\!\!\sim\;\ra\fg^{nil}(F^{ur})/\!\!\sim$. Note that the example in \cite{Tsa23} shows that an element in $\WF(\gamma)$ needs not match with an element in $\WF^{ur}(\gamma)$.

\begin{remark} 	
	In the example when $G=U_2$ is an unramified unitary group, $\fg^{nil}(F)/\!\!\sim_{semi}$ has two regular orbits while $\fg^{nil}(F^{ur})/\!\!\sim$ has only one. In general, thanks to Corollary \ref{cor:extension} there is an obvious map from $\fg^{nil}(F)/\!\!\sim_{semi}$ to $\fg^{nil}(F^{ur})/\!\!\sim$. Unfolding definition we see that $\WF^{ur}(\gamma)$ is again the set of maximal elements among the image of $\WF^{semi}(\gamma)$. 
\end{remark}

We have the following variant of Corollary \ref{cor:CO5} whose proof is identical.

\begin{corollary}\label{cor:CO5again} Suppose $r\in\Z$. Then we have $\cO(\sigma,c)=\cO(\sigma',c)$ is contained in the closure of one of $\cO^{semi}(\sigma',N^{\rho}_i(c'))$ ($1\le i\le\ell$).
\end{corollary}

\begin{theorem}\label{thm:goodness} (Corollary \ref{cor:good1}) Suppose $r\in\Z$ so that $m=1$. Fix $\sigma'\in\Phi^G(r)$ horizontal and $(x,r)\in\sigma'\cap\cA^a_{(p)}$. Let $c'=c_s'+c_n'\in{}^{\rho}_a\fg_x^{\heartsuit}=\fg_x^{\heartsuit}$ and $N(c')$ be as before. Then there exists $c\in N(c')(k)$ such that for any $\lambda:\mathbb{G}_m\ra\sG_x^{\heartsuit}$ and $\ell\in\Z_{>0}$ with $c\in{}^{\lambda}_{\ell}(\fg_x^{\heartsuit})$, we have $c'\in c+{}^{\lambda}_{<\ell}(\fg_x^{\heartsuit})$.
\end{theorem}

With the notations in Theorem \ref{thm:goodness}, up to conjugating by an element in $\sG_x^{\heartsuit}(k)=\sG_x(k)$ we may assume that $\lambda$ has image in $\sS$ and thus corresponds to $\til{\lambda}\in X_*(S)$. We remark that such $\lambda$ always exists, e.g. the one associated to an $\slt$-triple completing $c$ with $\ell=2$. For any such $\lambda$ and $\til{\lambda}$, there exists $\sigma\in\Phi^G$ and $\delta>0$ such that for $0<s<\delta$ we have $(x,r)-(s\til{\lambda},\ell s)\in\sigma$. This leads to

\begin{corollary}\label{cor:goodness} Let $c$ be as in Theorem \ref{thm:goodness}. If in the construction of the graph $C_G$ in (\ref{eq:defedge1}) we make the choice $\sigma'$ for $\sigma$ right after Definition \ref{def:below}. Then there is an edge from $\cO(\sigma,c)$ to $\cO(\sigma',c')$. We have $\cO^{semi}(\sigma,c)=\cO^{semi}(\sigma',N(c'))$.
\end{corollary}

\begin{proof} The setting in this subsection applies so that we have an edge from $(\sigma,c)$ to $(\sigma',c')$. By construction $c'$ and $N(c)$ are in the same $\Ad(\sG_x^{\heartsuit})$-orbit and therefore $\cO^{semi}(\sigma',c')=\cO^{semi}(\sigma',N(c))$.
\end{proof}

\begin{remark} Although the statement of Corollary \ref{cor:goodness} seems to rely on some choice, it implies that whenever $(\sigma',c')\in c_G(\gamma)$, then $\WF(\gamma)$ must contain an orbit in $\cO^{semi}(\sigma,c)$ or some larger orbit, because Theorem \ref{thm:graph} holds for any choice. Frequently in examples, Theorem \ref{thm:graph} will reversely force the choice in Corollary \ref{cor:goodness} to be the unique one because the wave-front set in the end does not depend on the choice.
\end{remark}

\begin{remark} Corollary \ref{cor:test} and \ref{cor:test2} will allow us to reduce the geometric problem - Problem \ref{prob:hard} - to the case when $\lambda$ is associated to an $\slt$-triple completing $c$. Even with the reduction, the problem seems rather formidable in general. For example, the construction of \cite{Tsa24} is based on some rather specific computer experiment of this problem (in \cite[Lemma 2.2, 2.3, 2.4]{Tsa24}) when $m=2$, $\sG_x^{\heartsuit}=GL_n$ and the action $\rho$ is given by $g\mapsto (g^t)^{-1}$. In this case Problem \ref{prob:hard} involves at least some geometry of Jacobians of hyperelliptic curves studied in \cite{Tsa18} based on \cite{Wan18}.
	
The author also gave some discussion \cite[\S4]{Tsa24} that this problem might be related to graded Springer theory (which is still largely being developed). Meanwhile, Corollary \ref{cor:good1} is concerned with the ``ungraded'' case $m=1$, namely the case when graded Springer theory becomes the well-understood generalized Springer theory. Indeed, Theorem \ref{thm:good1} behinds Corollary \ref{cor:good1} can be proved with Springer theory, though we take an elementary approach with a key step Proposition \ref{prop:Spr} inspired by Springer theory.
\end{remark}

\section{Examples}\label{sec:unram}

We now collect all our results so far to produce precise results about semi-unramified wave-front set (Definition \ref{def:semiunram}, an invariant stronger than $\WF^{ur}(\gamma)$) in some ``unramified'' situation. Let $\gamma=\gamma_r+\gamma_{>r}$ be as in \S\ref{sec:RedtoLevi}. We remark that this implies that $\gamma_r\in\Lie T$ for some torus $T\subset G$ that splits over an unramified extension. Again write $H=Z_G(\gamma_r)$. Our goal is to compute $\WF^{semi}(\gamma)$, using $c_{H,r}(\gamma_{>r})$ as a black box. The set $c_{G,r}(\gamma)$ is determined by $c_{H,r}(\gamma_{>r})$ via (\ref{eq:descent}). To apply Proposition \ref{prop:wf2} we need to determine $c_{G,s}(\gamma)$ for any slightly smaller $s$. Indeed, take $s=r-\epsilon$ with $\epsilon\in\Z_{(p)}$, $\epsilon>0$ and small enough so that there is no horizontal facet in $\Phi^G$ of depth within $[s,r)$, i.e. all $\sigma\in\Phi^G(s)$ are non-horizontal with $\depth(\sigma)\ge r$. Hence by Theorem \ref{thm:graph}, $c_{G,s}(\gamma)$ are precisely those $(\sigma,c)$ with an edge to $(\sigma',c')\in c_{G,r}(\gamma)$ in the manner of (\ref{eq:defedge1}). In this case $\sigma'$ is the horizontal augmented facet below $\sigma$ and the datum $(\sigma,c,\sigma',c')$ satisfies the setup in (\ref{eq:MP<2}), Problem \ref{prob:hard} and the paragraph in between. We can apply Corollary \ref{cor:CO5}, which shows that $\cO(\sigma,c)$ is contained in the $p$-adic closure of $\cO^{ur}(\sigma',N(c'))$. This gives
\begin{equation}\label{eq:urWFineq}
\renewcommand{\arraystretch}{1.5}
\begin{array}{rcl}
	\WF^{ur}(\gamma)&=&\text{the set of maximal elements in }\{\cO^{ur}(\sigma,c)\;|\;(\sigma,c)\in c_{G,s}(\gamma)\}\\
	&\le&\text{the set of maximal elements in } \{\cO^{ur}(\sigma',N(c'))\;|\;(\sigma',c')\in c_{G,r}(\gamma)\}.
\end{array}
\end{equation}
where $\le$ means that every element in $\WF^{ur}(\gamma)$ is $\le\cO$ for some $\cO$ in the RHS. When we impose the additional assumption that $r\in\Z$, we can replace $\cO^{ur}$ by $\cO^{semi}$ using Corollary \ref{cor:CO5again}. Corollary \ref{cor:goodness} says for every $(\sigma',c')\in c_{G,r}(\gamma)$, we can choose to have $(\sigma,c)$ as above with an edge to $(\sigma',c')$ such that $\cO^{semi}(\sigma,c)=\cO^{semi}(\sigma',c)$ is precisely $\cO^{semi}(\sigma',N(c'))$. Therefore when $r\in\Z$ we have
\begin{equation}\label{eq:urWF}
	\WF^{semi}(\gamma)=\text{the set of maximal elements in } \{\cO^{semi}(\sigma',N(c'))\;|\;(\sigma',c')\in c_{G,r}(\gamma)\}.
\end{equation}

We apply (\ref{eq:urWF}) to a number of cases.

\begin{example}\label{ex:toral} Let $\gamma\in\fg(F)$ be a good element of depth $r$. We don't assume $r\in\Z$ yet. We would like to compute $\WF(\gamma)$. Again write $H=Z_G(\gamma)$. We have $\gamma=\gamma+0$ is the decomposition in \S\ref{sec:RedtoLevi} and
\begin{equation}\label{eq:wf0}
	c_{H,r}(0)=\{(\varsigma_H,0)\;|\;\sigma_H\in\Phi^H(r)\}
\end{equation}
Denote by $\gamma_r$ the image of $\gamma$ to the sub-quotient $Z(\fh)(F)_{=r}$. Note that $Z(\fh)(F)_{=r}$ is a subspace of $\fg(F)_{x=r}$ for any $x\in\cB(H,F)$, so we can realize $\gamma_r$ in $\fg(F)_{x=r}$ for any $x\in\cB(H,F)$. By (\ref{eq:descent}) and (\ref{eq:wf0}) we have
\begin{equation}\label{eq:cgr-good}
	c_{G,r}(\gamma)=G(F).\{(\varsigma,\gamma_{=r})\;|\;\varsigma\in\Phi^G(r),\;\varsigma\cap\cB(H,F)\not=\emptyset\}.
\end{equation}
We plug (\ref{eq:cgr-good}) to the RHS of (\ref{eq:urWF}). We remark that since (\ref{eq:cgr-good}) is apparently $G(F)$-invariant, we can discard the $G(F)$-action part in (\ref{eq:cgr-good}) to make it simpler. For example, when $H/Z(G)$ is anisotropic, we have that $\cB(H,F)$ is a unique point modulo the central direction of $\cB(G,F)$. By Corollary \ref{cor:KM}, this makes $\cB(H,F)$ a single $r$-facet $\sigma'$ in $\cB(G,F)$. When $r\in\Z$, (\ref{eq:urWF}) gives
\begin{equation}\label{eq:CO6}
\WF^{semi}(\gamma)=\cO^{semi}(\sigma',N(\gamma_{=r}))
\end{equation}
which recovers \cite[Cor. 6.1]{CO23}. In general when $r\not\in\Z$, $\WF^{ur}(\gamma)$ is the set of maximal elements among $\cO^{ur}(\sigma,c)$ for those $(\sigma,c)$ satisfying Problem \ref{prob:hard} where we have a non-trivial $\rho$.
\end{example}

\begin{example}\label{ex:u6} Denote by $\varpi_F\in F$ be a fixed uniformizer. Let $G=U_6$ be an unramified unitary group over $F$ that splits over the quadratic unramified extension $E/F$, given by the following chosen hermitian form:
\begin{equation}\label{eq:hermitianform2}
	\langle x,y\rangle=x_1\bar{y}_1+x_2\bar{y}_2+x_3\bar{y}_3+x_4\bar{y}_4+x_5\bar{y}_5+\varpi_Fx_6\bar{y}_6.
\end{equation}
Note that there is an additional coefficient $\varpi_F$ for the 6th coordinate, making $G$ non-quasi-split. Fix $\lambda_1,\lambda_2,...,\lambda_6\in\cO_E^{\times}$ with distinct non-zero residues such that $\Tr_{E/F}\lambda_i=0$. Consider the following diagonal elements
\[
\gamma_{-1}=\varpi_F^{-1}\cdot\matr{
	0&&&&&\\
	&0&&&&\\
	&&\lambda_3&&&\\
	&&&\lambda_4&&\\
	&&&&\lambda_5&\\
	&&&&&\lambda_6
},
\]
\[
\gamma_0=\matr{
	\lambda_1&&&&&\\
	&\lambda_2&&&&\\
	&&0&&&\\
	&&&0&&\\
	&&&&0&\\
	&&&&&0
}
\]
\end{example}

\begin{proposition}\label{prop:u6} In Example \ref{ex:u6}, $\WF^{semi}(\gamma_{-1}+\gamma_0)$ contains of a nilpotent orbit of Jordan type $[4,1,1]$ and another of Jordan type $[3,3]$. 
\end{proposition}

This corresponds, for example, to a tame supercuspidal $\pi\in\Irr(G(F))$ whose Yu datum \cite[\S3]{Yu01} has $d=1$, $\vec{G}=(H,G)$ where $H=Z_G(\gamma_{-1})$, and $r_0=r_1=1$. That is, $\WF(\pi)$ consists of a nilpotent orbit of Jordan type $[4,1,1]$ and another of Jordan type $[3,3]$.

\begin{proof}[Proof of Proposition \ref{prop:u6}]
Note that $H=Z_G(\gamma_{-1})$ is a quasi-split $U_2$ times four copies of $U_1$. By the algorithm in \S\ref{sec:RedtoLevi}, we have to begin with $c_{H,{-1}}(\gamma_0)$. The set $\Phi^H(-1)$ of $(-1)$-facets is just the set of usual facets in the building $\cB(H,F)$, which is the same as in the building for $SU_2\cong SL_2$. The building has three $H(F)$-orbits of facets: the two orbits of hyperspecial vertices and the orbit of alcoves. Let $\{e_1,...,e_6\}$ be the basis of our hermitian space on which we equip the hermitian form (\ref{eq:hermitianform2}). Suppose $\{e_1',e_2'\}$ is a different basis of $\operatorname{span}_E(e_1,e_2)$ such that with respect to $\{e_1',e_2',e_3,...,e_6\}$ the form (\ref{eq:hermitianform2}) becomes
\begin{equation}\label{eq:hermitianform3}
	\langle x,y\rangle=\varpi_Fx_1\bar{y}_1+\varpi_Fx_2\bar{y}_2+x_3\bar{y}_3+x_4\bar{y}_4+x_5\bar{y}_5+\varpi_Fx_6\bar{y}_6.
\end{equation}
Let $H(F)_{y\ge 0}$ be the hyperspecial subgroup in $H(F)$ that stabilizes $\operatorname{span}_{\cO_E}(e_1,e_2)$ and $H(F)_{z\ge 0}$ be the hyperspecial subgroup that stabilizes $\operatorname{span}_{\cO_E}(e_1',e_2')$. These are exactly the two $H(F)$-orbits of hyperspecial subgroups. Also let $\varsigma\subset\cB(H,F)$ be any alcove. One checks that
\[
c_{H,-1}(\gamma_0)=H(F).\{((y,-1),0),((\varsigma,-1),0),((z,-1),c_n)\}
\]
where $c_n\in\fh(F)_{z=-1}$ is the unique $H(F)_{z\ge 0}$-orbit of non-trivial (regular) nilpotent elements. As before by $\gamma_{=-1}$ the image of $\gamma_{-1}$ in $Z(\fh)(F)_{=-1}\subset\fg(h)_{x=-1}\subset\fg(F)_{x=-1}$ for any $x\in\cB(H,F)$. Thanks to (\ref{eq:urWF}), we have
\begin{equation}\label{eq:wfu6}
	\renewcommand{\arraystretch}{1.5}
	\begin{array}{l}
		\WF^{semi}(\gamma)=\text{ the set of maximal elements in }\\
		\{\cO^{semi}((y,-1),N(\gamma_{-1})),\;\;\cO^{semi}((\varsigma,-1),N(\gamma_{-1})),\;\;\cO^{semi}((z,-1),N(\gamma_{-1}+c_n))\}.
	\end{array}
\end{equation}

We have to compute the groups $\sG_x^{\heartsuit}=\sG_x$ for $x=y,z$ and $x\in\varsigma$. This is an exercise about computing the embedding of buildings $\cB(H,F)\ra\cB(G,F)$. The result is quite intuitive:
\begin{enumerate}[label=(\roman*)]
	\item Thanks to (\ref{eq:hermitianform2}), $G(F)_{y\ge 0}$ is the parahoric subgroup of unitary operators stabilizing $\operatorname{span}_{\cO_E}(e_1,...,e_6)$ and $\operatorname{span}_{\cO_E}(e_1,e_2,...,e_5,\varpi_F^{-1}e_6)$. The reductive quotient $\sG_y\cong U_5\times U_1/_k$.
	\item Thanks to (\ref{eq:hermitianform3}), $G(F)_{z\ge 0}$ is the parahoric subgroup of unitary operators stabilizing $\operatorname{span}_{\cO_E}(e_1,...,e_6)$ and $\operatorname{span}_{\cO_E}(\varpi_F^{-1}e_1,\varpi_F^{-1}e_2,e_3,e_4,e_5,\varpi_F^{-1}e_6)$. The reductive quotient $\sG_z\cong U_3\times U_3/_k$.
	\item $G(F)_{\varsigma\ge 0}$ is a parahoric subgroup whose reductive quotient is $U_3\times U_1^3$.
\end{enumerate}
When $x=y$, the element $\gamma_{=-1}$ is a semsimple element in $\fg_y^{\heartsuit}(k)=\fg(F)_{y=-1}\cong\Lie(U_5\times U_1)(k)$ whose $\Lie U_5$-component has four distinct eigenvalues and one of them has multiplicity two. We denote nilpotent geometric orbits in $\fg_y^{\heartsuit}$ by a pair of partitions of $5$ and $1$ respectively. For example the regular nilpotent orbit is $([5],[1])$ while the zero orbit is $([1,1,1,1,1],[1])$. Then $N(c')$ corresponds to $([4,1],[1])$. Under $\sG_y^{\heartsuit}=U_5\times U_1\ira U_6=\sG_y^{\clubsuit}$ it goes to the orbit corresponding to $[4,1,1]$. By Proposition \ref{prop:samelift}, $\cO^{semi}((y,-1),N(\gamma_{=-1}))$ has Jordan type $[4,1,1]$ as well.

On the other hand, when $x=z$ we have that $\gamma_{=-1}+c_n\in\fg_z^{\heartsuit}(k)\cong\mathfrak{u}_3\times\mathfrak{u}_3$ is regular and thus by Corollary \ref{cor:edgecond1} $N(c')\subset\fg_z^{\heartsuit}(\bar{k})$ is again the regular nilpotent orbit, i.e. it corresponds to the partition $([3],[3])$. Under $\sG_z^{\heartsuit}\ira \sG_z^{\clubsuit}$ which is $U_3\times U_3\ira U_6$, it goes to the orbit corresponding to $[3,3]$. By Proposition \ref{prop:samelift}, $\cO^{semi}((z,-1),N(\gamma_{=-1}+c_n))$ has Jordan type $[3,3]$.

Lastly, as $\fg_{\varsigma=-1}$ has reductive quotient $\Lie(U_3\times U_1^3)$, $\cO^{semi}((\varsigma,-1),N(\gamma_{=-1}))$ is of type as most (in fact equal to) $[3,1,1,1]$. By (\ref{eq:wfu6}) we have that $\WF^{semi}(\gamma)$ has an element of type $[4,1,1]$ and another of type $[3,3]$.
\end{proof}

More generally than in Example \ref{ex:u6}, consider $\gamma=\gamma_{d-1}+\gamma_{d-2}+...+\gamma_0+\gamma_{-1}$ where
\begin{enumerate}
	\item $[r_i,r_j]=0$ for $-1\le i,j\le d-1$
	\item For $0\le i\le d-1$, $\gamma_i$ is good of depth $r_i\in\Z$ for a sequence of {\bf integers} $r_{d-1}<r_{d-2}<...<r_0$.
	\item $\gamma_{-1}$ is nilpotent. (It might be less distracting to start with the case $\gamma_{-1}=0$.)
\end{enumerate}
An element $\gamma$ can be written into this form iff its semisimple part lies in the Lie algebra of an unramified subtorus. Define $G^d:=G$ and inductively $G^i:=Z_{G^{i+1}}(\gamma_i)$ for $i=d-1,d-2,...,0$. The notations are chosen so that the sequence $G^0\subsetneq G^1\subsetneq...\subsetneq G^d=G$ plays the same role as in \cite{Yu01}. We would like to compute $\WF^{semi}(\gamma)$. Write $\gamma_i^{\dagger}:=\gamma_i+\gamma_{i-1}+...+\gamma_{-1}$ for $i=-1,0,...,d-1$. Following our algorithm, for $i=0,1,...,d-1$ we should inductively
\begin{enumerate}[label=(\roman*)]
	\item\label{step:red} From $c_{G^i,r_i}(\gamma_{i-1}^{\dagger})$ we deduce $c_{G^{i+1},r_i}(\gamma_i^{\dagger})$ via \S\ref{sec:RedtoLevi}.
	\item\label{step:non-nil} From $c_{G^{i+1},r_i}(\gamma_i^{\dagger})$ we deduce $c_{G^{i+1},r_i-\epsilon}(\gamma_i^{\dagger})$ as in \S\ref{subsec:edge1}.
	\item\label{step:nil} Use the algorithm in \S\ref{sec:graph} to further compute $c_{G^{i+1},r_{i+1}}(\gamma_i^{\dagger})$. This means we backtrack possible paths in the graph $C_{G^{i+1}}$. (This step is not needed for $i=d-1$ though it does not harm too.)
\end{enumerate}
For step \ref{step:non-nil}, Corollary \ref{cor:CO5} offers an upper bound for $c_{G^{i+1},r_i-\epsilon}(\gamma_i^{\dagger})$ for which Corollary \ref{cor:goodness} offers some similar (but different) existence result. On the other hand, in step \ref{step:nil} we have nilpotent $c'$ which can be simpler. However, it involves tracing paths in $C_{G^{i+1}}$ from depth $r_i-\epsilon$ to depth $r_{i+1}$, hence the ``combinatorial problem.''

We propose an estimate to the result of step \ref{step:non-nil} and \ref{step:nil} altogether. In step \ref{step:non-nil}, we already obtain the semi-unramified wave-front set $\WF^{semi,G^{i+1}}(\gamma_i^{\dagger})$ thanks to Theorem \ref{thm:goodness} and Appendix \ref{sec:gradedSpringer}. Then, Proposition \ref{prop:wf2} can be used backward to bound $c_{G^{i+1},r_{i+1}}(\gamma_i^{\dagger})$. That is
\begin{equation}\label{eq:upperbound}
c_{G^{i+1},r_{i+1}}(\gamma_i^{\dagger})\subset\{(\sigma,c)\;|\;\sigma\in\Phi^{G^{i+1}}(r_{i+1}),\;\cO^{semi}(\sigma)(\sigma,c)\le\cO\text{ for some }\cO\in\WF^{semi,G^{i+1}}(\gamma_i^{\dagger})\}
\end{equation}
We should have plugged $c_{G^{i+1},r_{i+1}}(\gamma_i^{\dagger})$ to step \ref{step:red} for $i=i+1$ and continue. Had we instead plug the RHS of (\ref{eq:upperbound}), we would get a potentially larger $c_{G^{i+2},r_{i+1}}(\gamma_{i+1}^{\dagger})$. We can continue the algorithm, using (\ref{eq:upperbound}) to replace step \ref{step:non-nil} and \ref{step:nil} for each $i=0,...,d-2$. By doing so we eventually get an upper bound for $\WF^{semi,G^d}(\gamma_{d-1}^{\dagger})=\WF^{semi,G}(\gamma)$. We propose that

\begin{conjecture}\label{conj:upperbound} This upper bound is always equal to $\WF^{semi,G}(\gamma)$.
\end{conjecture}


The rather few reasons we have to believe Conjecture \ref{conj:upperbound} is that it is true in Example \ref{ex:u6} and a few similar examples, and some vague intuition from the geometry of affine Springer fibers (which a large part of our algorithm from Waldspurger can be used to study too; the idea was discussed in \cite{Tsa15}, and many improvements and results in our \S\ref{sec:graph} and \S\ref{sec:geometric} applies to the method there). Conjecture \ref{conj:upperbound} is very close to saying that
\[
\WF^{semi,G}(\gamma)=\WF^{semi,G}(\gamma_{d-1}+\gamma_{d-2}^{\dagger})\text{ is determined by }\gamma_{d-1}\text{ and }\WF^{semi,G^{d-1}}(\gamma_{d-2}^{\dagger}).\]
However, as discussed in the introduction, this is {\bf not} true at least without the assumption that the semisimple part $\gamma$ lies in an unramified torus; see Example \ref{ex:WFnotinductive}.

\begin{remark} The example in \cite{Tsa24} and Example \ref{ex:u6} are sort of two extreme examples for which $\WF(\pi)$ contains two elements whose geometric orbits are incomparable. In the case of \cite{Tsa24}, our element $\gamma$ is what is denoted by $\til{A}$ in [\S4.1, {\it op. cit.}]. It is a good element of depth $-\frac{1}{2}$ with $H=Z_G(\gamma)$ an anisotropic torus so that the discussion in Example \ref{ex:toral} applies and the only problem is the geometric problem. On the other hand, in Example \ref{ex:u6} the geometric problem is much simpler in the sense that the grading $\rho$ is trivial, and one might say that the complication of Example \ref{ex:u6} is combinatorial. At the time of writing \cite{Tsa24} the author had expected, apparently wrongly, that the geometric problem is the only source for the geometric wave-front set to be not a singleton. Maybe a correct philosophy is that ``a random complicated enough (supercuspidal) representation has a non-singleton geometric wave-front set.''
\end{remark}

\begin{example}\label{ex:WFnotinductive} Let $E/F$ be a ramified quadratic extension with common residue field $k=\F_{23}$ and $G=U_7/_F$ splitting over $E$ be given by the hermitian form
	\begin{equation}\label{eq:hermitianform}
		\langle x,y\rangle=x_1\bar{y}_7+...+x_7\bar{y}_1,\;x_i,y_j\in E
	\end{equation}
	where $\bar{y}$ denotes the $E/F$-conjugate of $y$. Fix $\varpi_E\in E$ an uniformizer with $\bar{\varpi}_E=-\varpi_E$. Consider $\gamma=\gamma_0+\gamma_{1/2}\in\fg(F)$ with
	\[
	\gamma_0=\matr{
		0&&\til{5}\varpi_E\\
		&\mathbf{0}&\\
		\varpi_E^{-1}&&0
	}
	\]
	where $\mathbf{0}$ in the middle indicates a $5\times 5$ zero matrix, and $\til{5}\in F^{\times}$ is any element whose image in $k=\F_{23}$ is $5$, or really any non-square. Also
	\begin{equation}\label{eq:gamma1}
		\gamma_{1/2}=\varpi_E\cdot\matr{
			0&&&&&&0\\
			&0&3&&1&&\\
			&1&0&&&1&\\
			&&1&0&&&\\
			&&&1&0&3&\\
			&&&&1&0&\\
			0&&&&&&0
		},\;\;
		\gamma_{1/2}'=\varpi_E\cdot\matr{
		0&&&&&&0\\
		&0&1&&1&&\\
		&1&0&&&1&\\
		&&1&0&&&\\
		&&&1&0&1&\\
		&&&&1&0&\\
		0&&&&&&0
		}.
	\end{equation}
	The centralizer $H:=Z_G(\gamma_0)$ is the product of a $U_5$ (again splits over $E$) and an anisotropic torus $T$. The torus $T$ has rank $2$ and its reductive quotient $\mathsf{T}/_k$ is non-split and of rank $1$. There is a special vertex $y\in\cB(H,F)$ for which 
	\begin{enumerate}
		\item The stabilizer of $y$ in $U_5(F)$ is the stabilizer of the lattice $\cO_E^5$,
		\item $y$ is the unique point in $\cB(Z_H(\gamma_{1/2}),F)\ira\cB(H,F)$,
		\item $\sH_y\cong SO_5\times\sT$. 
	\end{enumerate}
	From this we see that there exists $(y,c_y)\in c_{H,0}(\gamma_{1/2})\cap c_{H,0}(\gamma_{1/2}')$ such that $\cO(y,c_y)\in\Lie U_5$ has Jordan type $[5]$; indeed, this comes from that for the $U_5$-part we have
	\[
	\Ad(\matr{\varpi_E^2&&&&\\&\varpi_E&&&\\&&1&&\\&&&-\varpi_E^{-1}&\\&&&&\varpi_E^{-2}})\gamma_{1/2}=\matr{0&&&&\\1&0&&&\\&1&0&&\\&&-1&0&\\&&&-1&0}+R
	\]
	where $R\in\Lie U_5(F)$ has all its entries in $\fm_E=\varpi_E\cO_E$. Consequently $\WF^H(\gamma_{1/2})$ and $\WF^H(\gamma_{1/2})$ are both the singleton set of the regular nilpotent orbit in $U_5$. On the other hand, $\cB(H,F)$ has another special vertex $z$ for which $\sH_z\cong Sp_4\times\sT$. For any regular nilpotent $c_z\in\Lie\sH_z(k)=\fh(F)_{z=0}$, we have that $\cO(z,c_z)\in\Lie U_5$ has Jordan type $[4,1]$. The highlight is that
	\begin{proposition}\label{prop:curve} For any regular nilpotent $c_z\in\Lie\sH_z(k)=\fh(F)_{z=0}$, we have $(z,c_z)\not\in c_{H,0}(\gamma_{1/2})$ while $(z,c_z)\in c_{H,0}(\gamma_{1/2}')$
	\end{proposition}

	With the proposition we can compute $\WF^{semi}(\gamma)$ using (\ref{eq:urWF}). We outline the result:
	\begin{enumerate}
		\item $y$ is a vertex in $\cB(G,F)$ with $\sG_y\cong SO_5\times Sp_2$ so that $\sH_y\ira\sG_y$ is $SO_5\times\sT\ira SO_5\times Sp_2$ given by $\sT\ira Sp_2$.
		\item $z$ is a vertex in $\cB(G,F)$ with $\sG_z\cong SO_1\times Sp_6$ so that $\sH_z\ira\sG_z$ is $Sp_4\times\sT\ira Sp_6$ (as a twisted Levi subgroup).
	\end{enumerate}
	Thus if $(z,c_z)\in c_{H,0}(\gamma_{1/2})$, $\cO^{semi}((z,0),N(\gamma_{=0}+c_z))$ will be an orbit of type $[6,1]$. Meanwhile $\cO(y,c_y)$ contributes an orbit in $\fg(F)$ of Jordan type $[5,2]$. All other facets also contribute orbits of type $<[5,2]$. Thus despite that $\WF^H(\gamma_{1/2})=\WF^H(\gamma_{1/2}')$ are both regular, $\gamma_{1/2}$ somewhat misses a subregular contribution, resulting in that $\WF^{ur,G}(\gamma_0+\gamma_{1/2})$ is smaller than $\WF^{ur,G}(\gamma_0+\gamma_{1/2}')$. It remains to prove Proposition \ref{prop:curve}.
	\end{example}

	\begin{proof}[Proof of Proposition \ref{prop:curve}] Suppose on the contrary that $(z,c_z)\in c_{H,0}(\gamma_{1/2})$. By our algorithm, namely Theorem \ref{thm:graph}, there exists a path \begin{equation}\label{eq:path}
	(\sigma_0,c_0)\ra(\sigma_1,c_1)\ra...(\sigma_{\ell-1},c_{\ell-1})\ra(\sigma_{\ell},c_{\ell})
	\end{equation}
	in $C_H$ where $(\sigma_0,c_0)=(z,c_z)$, $(\sigma_{\ell},c_{\ell})\in c_{H,\frac{1}{2}}(\gamma_{1/2})$, and the arrows indicate directed edges in $C_H$. Note that $c_0,...,c_{\ell-1}$ has to be nilpotent while $c_{\ell}$ will be the image of $\gamma_{1/2}$ and is necessarily non-nilpotent. We compute these $\sigma_i$ explicitly. Let $S\subset U_5\subset H$ be the diagonal torus. To be precise, $U_5(F)$ is the group of unitary operators on $E^5$ equipped with the hermitian form
	\begin{equation}\label{eq:hermitianform1-2}
	\langle x,y\rangle=x_1\bar{y}_5+...+x_5\bar{y}_1,\;x_i,y_j\in E.
	\end{equation}
	Let $S\subset U_5$ be the diagonal maximal $F$-split torus, i.e.
	\[
	S(F)=\{\matr{t_1&&&&\\&t_2&&&\\&&1&&\\&&&t_2^{-1}&\\&&&&t_1^{-1}}\;|\;t_1,t_2\in F^{\times}\}.
	\]
	Identify $X_*(S)\cong\Z^2$ with standard basis $e_1,e_2\in X_*(S)$ such that 
	\[
	e_1(t)=\matr{t&&&&\\&1&&&\\&&1&&\\&&&1&\\&&&&t^{-1}},\;\;e_2(t)=\matr{1&&&&\\&t&&&\\&&1&&\\&&&t^{-1}&\\&&&&1}.
	\]
	From now on we pretend that $H=U_5$ instead of $U_5\times T$, as the anisotropic factor $T$ will play no role. We identify the apartment corresponding to $S$ as $\cA\cong\R^2$, with the origin $(0,0)=y$, the point whose stabilizer in $U_5(F)$ is the stabilizer of $\cO_E^5\subset E^5$. At the same time, one can verify that that $H(F)$-orbit of $z$ is the $H(F)$-orbit of $(\frac{3}{4},\frac{1}{4})\in\cA$. Without loss of generality we suppose $z=(\frac{3}{4},\frac{1}{4})$. We have
	\[
	\fh(F)_{z=0}=\{\matr{a_5&a_3&0&a_2&a_1\\a_7&a_6&0&a_4&-a_2\\0&0&0&0&0\\a_9&a_8&0&-a_6&a_3\\a_{10}&-a_9&0&a_7&-a_5}\;|\;a_1\in\fm_E^{-3}/\fm_E^{-2},\;a_2\in\fm_E^{-2}/\fm_E^{-1},\;a_3,a_4\in\fm_E^{-1}/\cO_E,
	\]
	\[a_5,a_6\in\cO_E/\fm_E,\;a_7,a_8\in\fm_E/\fm_E^2,\;a_9\in\fm_E^2/\fm_E^3,\;a_{10}\in\fm_E^3/\fm_E^4\}
	\]
	Any regular nilpotent $c_z\in\fh(F)_{z=0}$ is conjugate under $\sH_z(k)$ to an element of the form
	\begin{equation}\label{eq:ca}
	c_a:=\matr{0&&&&\\\varpi_E&0&&&\\&&0&&\\&a&&0&\\&&&\varpi_E&0},\;a\in\fm_E/\fm_E^2.
	\end{equation}
	Let $\til{\lambda}\in X_*(S)$ be $\til{\lambda}=(-3,-1)$, and $\lambda\in X_*(\sS)$ the corresponding cocharacter to the reductive quotient. Then $c_a\in{}^{\lambda}_2\fg(F)_{(\frac{3}{4}-3s,\frac{1}{4}-s)=2s}$ for any $s\in\R$.
	\begin{lemma} In (\ref{eq:path}), we must have $\ell=10$. For each $0\le i\le\ell$, up to $H(F)$-conjugate we have $\sigma_i$ is given by the table below and $c_i=c_a$ for some $a\in\varpi_E\cO_E/\varpi_E^2\cO_E$ if $i<\ell$.
	\end{lemma}
\begin{center}
\renewcommand{\arraystretch}{1.5}
	\begin{tabular}{|c|l|}
		\hline $0\le i\le 10$&$\sigma_i$\\
		\hline $0$&$\sigma_0=\{z\}$, $z:=((\frac{3}{4},\frac{1}{4}),0)$\\
		\hline $1$&$\{((\frac{3}{4}-3s,\frac{1}{4}-s),2s)\;|\;0<s<\frac{1}{20}\}$\\
		\hline $2$&$\sigma_2=\{x_2\}$, $x_2:=((\frac{3}{5},\frac{1}{5}),\frac{1}{10})$\\
		\hline $3$&$\{((\frac{3}{4}-3s,\frac{1}{4}-s),2s)\;|\;\frac{1}{20}<s<\frac{1}{12}\}$\\
		\hline $4$&$\sigma_4=\{x_4\}$, $x_4:=((\frac{1}{2},\frac{1}{6}),\frac{1}{6})$\\
		\hline $5$&$\{((\frac{3}{4}-3s,\frac{1}{4}-s),2s)\;|\;\frac{1}{12}<s<\frac{1}{18}\}$\\
		\hline $6$&$\sigma_6=\{x_6\}$, $x_6:=((\frac{3}{8},\frac{1}{8}),\frac{1}{4})$\\
		\hline $7$&$\{((\frac{3}{4}-3s,\frac{1}{4}-s),2s)\;|\;\frac{1}{8}<s<\frac{3}{20}\}$\\
		\hline $8$&$\sigma_8=\{x_8\}$, $x_8:=((\frac{3}{10},\frac{1}{10}),\frac{3}{10})$\\
		\hline $9$&$\{((\frac{3}{4}-3s,\frac{1}{4}-s),2s)\;|\;\frac{3}{20}<s<\frac{1}{4}\}$\\
		\hline $10$&$\sigma_{10}=\{y\}$, $y=((0,0),\frac{1}{2})$\\
		\hline		
	\end{tabular}
\end{center}

\begin{proof} We prove the lemma step by step, i.e. by induction on $i$. The case $i=0$ is given. For each odd $i$, we are in the situation of (\ref{eq:defedge2}) analyzed in \S\ref{subsec:edge2}, where $\sigma=\sigma_{i-1}$, one computes that $\sigma^{\dagger}=\sigma_i$ is the same as in the above table. Also the construction gives $c_i=c_{i-1}$.
	
For even $i$, we are in the situation of (\ref{eq:defedge1}), for which $\sigma'=\sigma_i=\{x_i\}$ in the above table is the unique augmented facet below $\sigma=\sigma_{i-1}$ and thus the correct one. The challenge is that a priori we have many $c'$ in (\ref{eq:defedge1}). For $i=2,4,6,8$, we have to show that every $c'\in c_a+\fh(F)_{>\sigma_{i-1}}/\fh(F)_{>\sigma_i}$ is either in $\Ad(\sG_{x_i}(k))c_a$ or non-nilpotent (which we can discard). A useful input is Corollary \ref{cor:edgecond2}, which asserts that $\cO(\sigma_i,c_i)$ has to be increasing, and when $\cO(\sigma_{i-1},c_{i-1})=\cO(\sigma_i,c_i)$ we must have $c_i=c_{i-1}$ up to conjugate. Since by induction $\cO(\sigma_{i-1},c_i)=\cO(\sigma_{i-1},c_a)$ has Jordan type $[4,1]$, we only have to rule out the case when $\cO(\sigma_i,c_i)$ might have Jordan type $[5]$, which by Proposition \ref{prop:samelift} must also be the Jordan type of $c_i$ itself.

The rest is direct computation. For $i=4,6$, the Moy-Prasad quotient $\fh(F)_{\sigma_i}$ simply has no nilpotent element of Jordan type $[5]$, so we have nothing left to rule out. For $i=2,8$ there is. But in both cases the Moy-Prasad quotient looks like (using $i=2$ as example)
\[
\fh(F)_{\sigma_2}=\{\matr{&&b&&\\a_7&&&&\\&&&&b\\&a_8&&&\\&&&a_7&}\;|\;a_7,a_8\in\fm_E/\fm_E^2,\;b\in\fm_E^{-1}/\cO_E\}
\]
and $\fh(F)_{>\sigma_1}/\fh(F)_{>\sigma_2}$ is the ``$b$-part,'' i.e. those in $\fh(F)_{\sigma_2}$ with $a_7=a_8=0$. With $c_a$ as in (\ref{eq:ca}), we see that $c'\in c_a+\fh(F)_{>\sigma_1}/\fh(F)_{>\sigma_2}$ is either $c_a$ itself (which has type $[4,1]$) or non-nilpotent, in particular never nilpotent of type $[5]$. This finishes the proof.
\end{proof}

We continue the proof of Proposition \ref{prop:curve}. Thanks to the Lemma, it remains to look at the edge $(\sigma_9,c_a)\ra (\sigma_{10},c_{10})$. In $\sigma_{10}=(y,\frac{1}{4})$. The question is whether $c_a+\fh(F)_{>\sigma_9}/\fh(F)_{>\sigma_{10}}$ contains an element in $\Ad(\sH_y(k))\gamma_{=1/2}$ or $\Ad(\sH_y(k))\gamma_{=1/2}$, where $\gamma_{=1/2}$ (resp. $\gamma_{=1/2}$) is the reduction of $\gamma_{1/2}$ (resp. $\gamma_{1/2}'$) in $\fh(F)_{y=\frac{1}{2}}$. One verifies that $\fh(F)_{>\sigma_9}/\fh(F)_{>\sigma_{10}}\subset\fh(F)_{y=1/2}$ is of the form
\[
\{\matr{b_2&b_5&b_7&b_9&b_{10}\\&b_3&b_6&b_8&b_9\\&b_1&b_4&b_6&b_7\\&&b_1&b_3&b_5\\&&&&b_2}\;|\;b_i\in\varpi_E\cO_E/\varpi^2\cO_E\}
\]
With (\ref{eq:ca}), to say $\Ad(\sH_y(k))\gamma_{=1/2}$ does not intersect $c_a+\fh(F)_{>\sigma_9}/\fh(F)_{>\sigma_{10}}$ is to say that the affine variety
\begin{equation}\label{eq:curve}
	\{\bar{g}\in SO_5\;|\;\Ad(\bar{g})\matr{
		0&3&&1&\\
		1&0&&&1\\
		&1&0&&\\
		&&1&0&3\\
		&&&1&0
	}\in\matr{
		*&*&*&*&*\\
		\not=0&*&*&*&*\\
		0&*&*&*&*\\
		0&\not=0&*&*&*\\
		0&0&0&\not=0&*
	}\}
\end{equation}
has no $k$-point (resp. has a $k$-point with $3$ replaced by $1$ in (\ref{eq:curve})). Here $*$ means that the entry can be anything and $\not=0$ means that the entry has to be non-zero. This can really be viewed as a subvariety $SO_5/B$ as in \cite[App. A]{Tsa24}; in fact it's an affine curve in $SO_5/B$. That it doesn't have $k$-point (or it has) can be checked by a computer program as in {\it loc. cit.}. This finishes the analysis of Example \ref{ex:WFnotinductive}.
\end{proof}

\section{Wave-front sets of representations}\label{sec:rep}

\begin{remark}\label{rmk:wfsc} Let us sketch the relation between $\WF(\gamma)$ and $\WF(\pi)$ for regular supercuspidal representations $\pi$ \cite{Kal19}. Essentially the result is in \cite{KM06}, but maybe the best reference is \cite[Thm. 4.3.5]{FKS21}. Consider any $\gamma\in\fg^{rs}(F)$ that is regular semisimple and elliptic (i.e. $Z_G(\gamma)/Z(G)$ is anisotropic). Say $S=Z_G(\gamma)$. We assume (beyond Appendix \ref{app:char}) that the exponential map converges on $\ft(F)_{>r}$ giving $\ft(F)_{>r}\cong S(F)_{>r}$ for some $r>0$, and for any $\alpha\in\Phi(G,S)$ (over $F^s$) we have $\operatorname{val}(d\alpha(\gamma))<-r$ so that \cite[Thm. 4.3.5]{FKS21} applies\footnote{To the best of my knowledge, there is a typo in the statement of \cite[Thm. 4.3.5]{FKS21} that the condition $C^{\le0}_{G}(X)=S$ in the language of \cite[\S4]{Spi21} is missing. This should not affect their result nor our result; with large enough $p$ that condition for us is $\operatorname{val}(d\alpha(\gamma))\le 0$.} with their $X$ being our $\gamma$. We have a character
\[
\Theta_{s,\gamma}:S(F)_{>s}\ra\Cc,\;\;t\mapsto\psi(B(\log(t),\gamma))
\]
where $B(\cdot,\cdot)$ is a $G$-invariant non-degenerate symmetric bilinear form normalized as in \cite[\S2.5]{DK06} and $\psi:F\ra\Cc$ an additive character trivial on $\mathfrak{m}_F$ but non-trivial on $\cO_F$. Suppose $\theta:S(F)\ra\Cc$ is any character extending $\theta_{s,\gamma}$. Then any such $\theta$ is regular in the sense that \cite{Kal19} gives a class of regular supercuspidal representations $\pi$. For any of them, near the identity the character $\Theta_{\pi}$ is proportional to the Fourier transform $\hat{I}_{\gamma}$; see \cite[Thm. 4.3.5]{FKS21} where their $\gamma_0=\mathrm{id}$, their $J$ is our $G$, their $X$ is our $\gamma$ and their $\hat{O}^J_{X^g}=\hat{O}^G_{X}$ is our $\hat{I}_{\gamma}$. This gives $\WF(\pi)=\WF(\gamma)$.

Conversely, every regular supercuspidal $\pi$ is associated to some elliptic $\gamma\in\fg(F)$ (non-unique) so that $\WF(\pi)=\WF(\gamma)$. Hence the study of $\WF(\gamma)$ for elliptic $\gamma$ is more-or-less equivalent to the study of wave-front sets of regular supercuspidal representations. By LLC for tori, the character $\theta$ as above gives rises to $\varphi_{\theta}:W_F\ra{}^LS$. With any $L$-embedding ${}^LS\ira{}^LG$ constructed in \cite[\S5]{Kal19} we get a Langlands parameter for which the L-packet {\it loc. cit.} consists of regular supercuspidal representations whose wave-front sets are given by $\WF(\gamma')$ for some $\gamma'$ stably conjugate to our $\gamma$ above, with each $\gamma'$ appearing at least once.
\end{remark}

\begin{remark} In \cite[Thm. 1.4]{CJLZ23}, Cheng Chen, Dihua Jiang, Dongwen Liu and Lei Zhang achieved a significant step towards the conjecture \cite[Conj. 1.3]{JLZ22a} that describes the wave-front sets for representations in generic $L$-packets for classical groups via an algorithm regarding the Langlands parameter on the Galois side. We would love to compare the algorithm in \cite{JLZ22a} and our algorithm in the special case for regular supercuspidal representations, and learn from \cite{JLZ22a} how to analyze and/or simplify our algorithm.
\end{remark}

\begin{remark} Beyond regular supercuspidal representations, \cite[Thm. 10.2.2]{Spi21} actually reduces the character of a general supercuspidal representation $\pi$ of $G(F)$ to the character of a depth-zero supercuspidal representation $\pi_0$ of $G^0(F)$, a twisted Levi subgroup. One can define similar invariants $c_{G,r}(\pi)$ and $c_{G^0,s}(\pi_0)$ but only for $r<-\dep(\pi)$ and $s<0$, namely
\[
c_{G,r}(\pi)=\{(\sigma,c)\in C_{G,r}\;|\;\Theta_{\pi}(FT[c+\fg(F)_{>\sigma}])>0\}.
\]
Here $[c+\fg(F)_{>\sigma}]$ is the characteristic function of the coset and $FT$ denotes Fourier transform. Note that $\Theta_{\pi}(FT[c+\fg(F)_{>\sigma}])$ counts (up to normalization) the multiplicity of $c$-isotypic factors in the fixed subspace $\pi^K$ where $K={G(F)_{x\ge -r}}$ for $(x,r)\in\sigma$. Hence $\Theta_{\pi}(FT[c+\fg(F)_{>\sigma}])$ is always non-negative. With \cite[Thm. 10.2.2]{Spi21} and above definition, our algorithm again computes $c_{G,r}(\pi)$ in terms of $c_{G^0,s}(\pi_0)$. When $s=-\epsilon$ for very small $\epsilon>0$, $c_{G^0,s}(\pi_0)$ is determined by whether the cuspidal representation of the finite reductive group behind $\pi_0$ appears in the generalised Gelfand–Graev representations \cite[\S1.3]{Kaw85} associated to all nilpotent orbits, not necessarily the maximal orbits in the wave-front set.

In general, one can hope that a similar reduction exists from general irreducible admissible representations of $G(F)$ to depth-zero representations of some $G^0(F)$ for which some generalization of \cite[Thm. 10.2.2]{Spi21} holds. If this is indeed the case, then our algorithm could again compute $c_{G,r}(\pi)$ in terms of $c_{G^0,s}(\pi_0)$.
\end{remark}

\appendix
\section{Restrictions on $p$}\label{app:char}

Throughout this article we enforce the hypotheses in \cite[\S1.1]{Wa06}. In \cite[App. 3]{Wa06} the hypotheses were shown to hold when $p>6N-1$ where $N$ is the absolute rank and furthermore $p\ge 271$ if there is an exceptional factor. We remark that $p>6N-1$ does not hold in Example \ref{ex:WFnotinductive} but in that case it's not hard to verify the original hypotheses in \cite[\S1.1]{Wa06}.

We highlight that the hypotheses ensure a number of properties: 
\begin{enumerate}
	\item The exponential map is defined on any nilpotent element.
	\item The theory of $\slt$-triples always work well; we always get a semisimple representation of the $\slt$ and a homomorphism from the group $SL_2$.
	\item Any reductive group over the local field ever appears splits over a tame extension.
	\item All important coordinates on the Bruhat-Tits building have denominator coprime to $p$.
\end{enumerate}

\section{Graded Lie algebras and Springer theory}\label{sec:gradedSpringer}

In this appendix we use a different set of notations from the rest of this article. Let $k$ be a finite field with a fixed algebraic closure $\bar{k}$, again with assumption on $p$ as in Appendix \ref{app:char}. Fix a connected reductive group $G$ over $k$ and $\fg:=\Lie G$. Fix $m\in\Z_{>0}$ invertible in $k$ and $\rho:\mu_m/_k\ra\operatorname{Aut}(G)$ be a homomorphism of algebraic groups. For $a\in\Z/m$, we denote by $\fg_a$ the $k$-subspace of $\fg=\Lie G$ on which $\rho$ acts with weight $a$; they were denoted ${}^{\rho}_a\fg$ in \S\ref{sec:geometric}, but thankfully we will have less confusing subscripts in this appendix. Let also $G_0=(G^{\rho(\mu_m)})^o$. Let $\lambda:\mathbb{G}_m\ra G_0$ be any cocharacter. Fix $\ell\in\Z$, $a\in\Z/m$, and a nilpotent element $c\in{}^{\lambda}_{\ell}\fg_a(k)$. For any $i\in\Z$, $i<\ell$, choose $V_i\subset{}^{\lambda}_{i}\fg_a$ a $k$-subspace complementary to $[c,{}^{\lambda}_{i-\ell}\fg_0]$. Let $V=\bigoplus_{i<\ell}V_i$.

\begin{lemma}\label{lem:retract} For every $c'\in c+V$ we have that $c$ is in the Zariski closure of $\bar{k}\cdot\Ad(G_0(\bar{k}))c'$. If moreover $c$ is in $\Ad(G(\bar{k}))c'$, then $c'=c$.
\end{lemma}

\begin{proof} Consider the $\Gm$-action on $c+V$ given by $t.u=t^{-\ell}\Ad(\lambda(t))u$ for $t\in\bar{k}^{\times}, u\in c+V$. This $\Gm$-action retracts $c+V$ to $\{c\}$ and preserves $\bar{k}\cdot\Ad(G_0(\bar{k}))c'$. Hence if $c'\in c+V$ then $c$ is in the Zariski closure of $\bar{k}\cdot\Ad(G_0(\bar{k}))c'$. It remains to consider the case $c\in\Ad(G(\bar{k}))c'$, i.e. $c'\in\Ad(G(\bar{k}))c\cap(c+V)$. Then we also have $t^{-\ell}\Ad(\lambda(t))c'\in\Ad(G(\bar{k}))c\cap(c+V)$ for any $t\in\bar{k}^{\times}$. Suppose on the contrary $c'\not=c$. Then the Zariski closure of the infinite set
	\[
	\{t^{-\ell}\Ad(\lambda(t))c'\;|\;t\in\bar{k}^{\times}\}	
	\]
	is a positive dimensional subvariety of the Zariski closure of $\Ad(G(\bar{k}))c\cap(c+V)$, which is $\Ad(G(\bar{k}))c\cap(c+V)$ itself because any nilpotent element in $c+V$ lives in an equal or larger orbit. The existence of this positive dimensional subvariety of $\Ad(G(\bar{k}))c\cap(c+V)$ shows that the tangent space $T_c(\Ad(G(\bar{k}))c)$ and $T_c(c+V)=V$ has non-trivial intersection. But $T_c(\Ad(G(\bar{k}))c)=[c,\fg]$ while 
	\[[c,\fg]\cap V=[c,\fg_0]\cap V=[c,{}^{\lambda}_{<0}\fg_0]\cap V=0,\]
	a contradiction.
\end{proof}

Denote by ${}^{\lambda}_{<0}G_0\subset G_0$ the connected unipotent subgroup whose Lie algebra is ${}^{\lambda}_{<0}\fg_0$. 

\begin{lemma}\label{lem:test} Suppose $c'\in c+{}^{\lambda}_{<\ell}\fg_a(k)$. Then $c$ is in the Zariski closure of $\bar{k}\cdot\Ad(G_0(\bar{k}))c'$. If furthermore $c'\in\Ad(G(\bar{k}))c$, then $c'\in\Ad({}^{\lambda}_{<0}G_0(k))c$.
\end{lemma}

\begin{proof} We claim that any element in $c+{}^{\lambda}_{<\ell}\fg_a$ is in the same $\Ad({}^{\lambda}_{<0}G_0(k))$-orbit as some element in $c+V$. The claim together with Lemma \ref{lem:retract} finishes the proof. We use induction for $j\in\Z_{\ge1}$ on the statement that any element in $c+{}^{\lambda}_{<\ell}\fg_a$ is in the same $\Ad({}^{\lambda}_{<0}G_0(k))$-orbit as some element in $c+{}^{\lambda}_{\le\ell-j}\fg_a+V$. For $j$ large enough this proves the lemma. The statement is trivially true for $j=1$. For a general element in $c+{}^{\lambda}_{<\ell}\fg_a$ for which we already conjugate it by ${}^{\lambda}_{<0}G_0(k)$ to $c_j\in c+{}^{\lambda}_{\le\ell-j}\fg_a+V$, we would like to conjugate it again so that it lives in $c+{}^{\lambda}_{\le\ell-j-1}\fg_a+V$. Denote by ${}^{\lambda}_{\ell-j}(c_j)$ the projection of $c_j$ from $\fg_a=\bigoplus{}^{\lambda}_i\fg_a$ to ${}^{\lambda}_{\ell-j}\fg_a$.
By the construction of $V_{\ell-j}$, there exists an $y\in{}^{\lambda}_{-j}\fg_0(k)$ such that ${}^{\lambda}_{\ell-j}(c_j)+[y,c]\in V_{\ell-j}\subset{}^{\lambda}_{\ell-j}\fg_a$. This implies that $\exp(y)c_j\in c+{}^{\lambda}_{\le\ell-j-1}\fg_a+V$, as desired.
\end{proof}

\subsection{Existence of rational points}\label{subsec:goodness}

From now on until Proposition \ref{prop:restrict}, we will not make use of the action $\rho:\mu_m\ra\operatorname{Aut}(G)$. Consider $x\in\fg(k)$ with the Jordan decomposition $x=x_s+x_n$. We say $x$ {\bf splits} over a field extension $k'/k$ if $L:=Z_G(x_s)$ is a Levi subgroup over $k'$, i.e. $L$ is contained in a parabolic subgroup $P\subset G$ defined over $k'$, and just say $x$ splits if it does so over $k$. Apparently $x$ splits over $\bar{k}$. Let $P$ be as above (defined over $\bar{k}$) and $\fu_P$ be the Lie algebra of its unipotent radical. 
\begin{definition}\label{def:LS} \cite{LS79} We denote by $N(x)$ unique nilpotent $\Ad(G(\bar{k}))$-orbit whose intersection with $x_n+\fu_P(\bar{k})$ is a dense open subset of the latter.
\end{definition}
\noindent The orbit $N(x)$ is the Lusztig-Spaltenstein induction of $x_n$ from $L$ to $G$. Following \cite[\S4]{CO23}, consider the variety $\Theta(x)$ of $\mathfrak{sl}_2$-triples $(c,h,d)\in\fg(\bar{k})^3$ such that
\begin{enumerate}
	\item $c,d\in N(x)$, and
	\item $x\in c+Z_{\fg}(d)$ is in the so-called Slodowy slice.
\end{enumerate}
The first goal of this section is to show the following generalization of \cite[Lemma 4.5]{CO23}. 
\begin{theorem}\label{thm:good1} The connected centralizer $Z_G(x)^o$ acts transitively on $\Theta(x)$. Hence $\Theta(x)(k)$ is non-empty.
\end{theorem}

\begin{remark} The second sentence of the theorem follows from the first thanks to Lang's theorem. The result is essentially that every element is good in the sense of \cite[\S4]{CO23}. It also gives a different proof of \cite[Thm. 0.3]{Kat82} (also used in \cite[\S4]{CO23}), which is equivalent to that $\Theta(x)$ is a $Z_G(x)$-orbit.
\end{remark}

Let $C(\fg(k))$ be the space of $\C$-valued functions on $\fg(k)$. We fix a $G(k)$-invariant non-degenerate symmetric bilinear form $B(\cdot,\cdot):\fg\times\fg\ra\mathbb{G}_a$ and a non-trivial additive character $\psi:k\ra\Cc$. For $f,f'\in C(\fg(k))$, write
\[
\inn{f}{f'}:=\sum_{x\in\fg(k)}f(x)f'(x),\;\;\hat{f}(y):=\sum_{x\in\fg(k)}\psi(B(x,y))f(x).
\]
Let $J(\fg(k))\subset C(\fg(k))$ be the subspace of $\Ad(G(k))$-invariant functions, and $J(\fg(k))^{nil}\subset J(\fg(k))$ be the subspace of those supported on $\fg^{nil}(k)$, the nilpotent cone. The Fourier transform $f\mapsto\hat{f}$ maps $J(\fg(k))$ isomorphically to itself, and let $J(\fg(k))^{conil}$ be the image of $J(\fg(k))^{nil}$ under the Fourier transform. For any sub-$k$-variety $X\subset\fg$, let $I_X$ be the function that takes the value $1$ on $X(k)$ and zero elsewhere. The idea of the following proposition comes from Springer theory, particularly \cite[Thm. 4.4]{Spr76}. See also the MathOverflow question \cite{MO202007}.

\begin{proposition}\label{prop:Spr} Let $x$ be split with $L=Z_G(x_s)$ and $P$ a parabolic subgroup containing $L$ and defined over $k$. We have
	\begin{equation}\label{eq:Spr}
		|U_P(k)|\cdot\inn{I_x}{\xi}=\inn{I_{x_n+\fu_P}}{\xi},\;\forall \xi\in J(\fg(k))^{conil}.
	\end{equation}
\end{proposition}

\begin{proof}
	We have $I_{x+\fu_P}(y)=\sum_{u\in U_P(k)}I_x(\Ad(u)y)$ and thus
	\begin{equation}\label{eq:Spr1}
		|\fu_P(k)|\cdot\inn{I_x}{\xi}=\inn{I_{x+\fu_P}}{\xi},\;\forall \xi\in J(\fg(k)).
	\end{equation}
	We compare $I_{x+\fu_P}$ and $I_{x_n+\fu_P}$ by looking at Fourier transforms. The two functions are apparently invariant under translation by $\fu_P(k)$. Thus their Fourier transforms are supported on $\fp=\fl\oplus\fu_P$ where $\fp=\Lie P$ and $\fl=\Lie L$. Note that $x-x_n=x_s\in Z(\fl)$. Since $Z(\fl)$ is necessarily perpendicular to $\fl^{der}+\fu_P$ under our bilinear form $B(\cdot,\cdot)$, we have that $\hat{I}_{x+\fu_P}$ and $\hat{I}_{x_n+\fu_P}$ have the same restriction on $\fl^{der}(k)+\fu_P(k)$. Since all nilpotent elements in $\fp$ lies in $\fl^{der}(k)+\fu_P(k)$, this shows that $\hat{I}_{x+\fu_P}$ and $\hat{I}_{x_n+\fu_P}$ have the same restriction to the nilpotent cone $\fg^{nil}(k)$. This implies that
	\[
	\inn{\hat{I}_{x+\fu_P}}{\xi}=\inn{\hat{I}_{x_n+\fu_P}}{\xi},\;\forall \xi\in J(\fg(k))^{nil}.
	\]
	which by Fourier inversion is equivalent to
	\[
	\inn{I_{x+\fu_P}}{\xi}=\inn{I_{x_n+\fu_P}}{\xi},\;\forall \xi\in J(\fg(k))^{conil}
	\]
	which together with (\ref{eq:Spr1}) implies (\ref{eq:Spr}).
\end{proof}

By our assumption on $\operatorname{char}(k)$, for any $\mathfrak{sl}_2$-triples $(c,h,d)$ there is a cocharacter $\lambda:\mathbb{G}_m\ra G$ associated to $(c,h,d)$ such that $c\in {}^\lambda_2\fg$ (the weight $2$ subspace), $h\in {}^\lambda_0\fg$ and $d\in {}^\lambda_{-2}\fg$, and
\begin{enumerate}
	\item $\ad(c)$ maps ${}^\lambda_{\le-1}\fg$ injectively into ${}^\lambda_{\le1}\fg$.
	\item ${}^\lambda_{\le 1}\fg=Z_{\fg}(d)\oplus[c,{}^\lambda_{\le-1}\fg]$.
\end{enumerate}
Let ${}^\lambda_{\le -1}G$ be the unipotent subgroup of $G$ whose Lie algebra is ${}^\lambda_{\le -1}\fg$; we can also characterize ${}^\lambda_{\le -1}G$ as the unipotent subgroup of the parabolic subgroup determined by $\lambda$. The following lemma is somewhat standard and can be proved analogous to Lemma \ref{lem:test} (with extra care to prove the uniqueness part).

\begin{lemma}\label{lem:sl2-2} For every element $y\in c+{}^\lambda_{\le 1}\fg(k)$, there is a unique element $u\in ({}^\lambda_{\le -1}G)(k)$ such that $\Ad(u)y\in c+Z_{\fg}(d)$.
\end{lemma}

For any $\mathfrak{sl}_2$-triples $(c,h,d)$ denote by $f_{c,h,d}\in C(\fg(k))$ the function with $f_{c,h,d}(x)=\sum_{g\in G(k)}I_{c+Z_{\fg}(d)}(\Ad(g)x)$. 

\begin{lemma}\label{lem:conil} We have $f_{c,h,d}\in J(\fg(k))^{conil}$.
\end{lemma}

\begin{proof} As before let $I_{c+{}^\lambda_{\le 1}\fg}\in C(\fg(k))$ be the function that takes the value $1$ on $c+{}^\lambda_{\le 1}\fg(k)$ and $0$ elsewhere. Since $I_{c+{}^\lambda_{\le 1}\fg}$ is invariant under translation by ${}^\lambda_{\le 1}\fg(k)$, its Fourier transform $\hat{I}_{c+{}^\lambda_{\le 1}\fg}$ is supported on ${}^\lambda_{\le -2}\fg(k)$, in particular in $\fg^{nil}(k)$. Lemma \ref{lem:sl2-2} implies that 
	\[
	f_{c,h,d}(x)=\sum_{g\in G(k)}I_{c+Z_{\fg}(d)}(\Ad(g)x)=\frac{1}{|({}^\lambda_{\le -1}G)(k)|}\sum_{g\in G(k)}I_{c+{}^\lambda_{\le 1}\fg}(\Ad(g)x).\]
	The Fourier transform $\hat{\xi}$ is thus a sum of conjugates of $\hat{I}_{c+{}^\lambda_{\le 1}\fg}$, all supported in $\fg^{nil}(k)$, i.e. $\hat{f}_{c,h,d}\in J(\fg(k))^{nil}$ and thus $f_{c,h,d}\in J(\fg(k))^{conil}$.
\end{proof}

We also need more lemmas regarding $\slt$-triples:


\begin{lemma}\label{lem:sl2-4} The centralizer $Z_G(c)$ acts transitively on the set of $\slt$-triples $(c,h',d')$ completing $c$, and for any such we have $\pi_0(Z_G(c,h',d'))\xra{\sim}\pi_0(Z_G(c))$.
\end{lemma}

Lastly, let us recall the following which relies on a little bit of Weil Conjectures.

\begin{lemma}\label{lem:Weil} A variety $X$ over $k$ has dimension $\delta$ and only one geometric component of dimension $\delta$ if and only if there exists a constant $C$ and a finite extension $k^{\dagger}/k$ such that $|X(k')-|k'|^\delta|<C\cdot |k'|^{\delta-1}$ for any finite extension $k'/k^{\dagger}$.
\end{lemma}

For our convenience, let us denote the second condition in the lemma by $|X(k')|\sim|k'|^{\delta}$. To prove Theorem \ref{thm:good1}, we first prove a weaker version of it:

\begin{proposition}\label{prop:1comp} Let $c\in N(x)$. The variety $\Theta(x)$ has dimension $\delta_c:=\dim Z_G(c)/Z_G(c,h,d)$ and has only one geometric component of this dimension.
\end{proposition}

\begin{proof} Consider the case $x$ is nilpotent, i.e. $x=c$. Then $N(c)=\Ad(G(\bar{k}))c$. By Lemma \ref{lem:retract} applied to our case without $\rho$ (i.e. $\rho$ is trivial), $\Theta(c)$ is equal to the variety of $\slt$-triples $(c,h',d')$ completing $c$, on which $Z_G(c)$ acts transitively by Lemma \ref{lem:sl2-4}. Moreover, we have $\pi_0(Z_G(c,h',d'))\xra{\sim}\pi_0(Z_G(c))$ and that $Z_G(c)$ acts transitively on $\Theta(c)$ with stabilizer at $(c,h,d)\in\Theta(c)$ equal to $Z_G(c,h,d)$, i.e. we have $Z_G(c)/Z_G(c,h,d)\cong Z_G(c)^o/Z_G(c,h,d)^o\cong\Theta(c)$. Thanks to Lemma \ref{lem:Weil} this proves the proposition for $x=c$, i.e. $|\Theta(c)(k')|\sim|k'|^{\delta_c}$.

For $y\in\fg(k)$, denote by $\Xi_x(y)$ the set of $\slt$-triples $(c,h,d)$ with $c\in N(x)$ and such that $y\in c+Z_{\fg}(d)$. Hence $\Theta(x)=\Xi_x(x)$. The $k$-points in $N(x)$ is a finite number of $\Ad(G(k))$-orbits, say $c_1,c_2,..,c_m$. We can complete each of them to $\slt$-triples $(c_i,h_i,d_i)$. By Lemma \ref{lem:sl2-4}, $\Xi_x(y)(k)$ is the disjoint union of the set of $\Ad(G(k))$-conjugates $(c,h,d)$ of $(c_i,h_i,d_i)$ such that $x\in c+Z_{\fg}(d)$. Hence for any $y\in\fg(k)$ we have
\begin{equation}\label{eq:Xi}
|\Xi_x(y)(k)|=\sum_{i=1}^m\frac{1}{|Z_G(c_i,h_i,d_i)(k)|}\cdot\inn{I_y}{f_{c_i,h_i,d_i}}.
\end{equation}
For general $x$, the statement to prove is geometric and there is no harm to perform a base change on the ground field $k$. In particular, we may and shall assume that $x$ is split. Plug $y=x$ into (\ref{eq:Xi}) and apply Proposition \ref{prop:Spr} and Lemma \ref{lem:conil} to the RHS, we have
\begin{equation}\label{eq:count1}
|\fu_P(k)|\cdot\sum_{i=1}^m\frac{1}{|Z_G(c_i,h_i,d_i)(k)|}\cdot\inn{I_x}{f_{c_i,h_i,d_i}}=
\sum_{u\in\fu_P(k)}\sum_{i=1}^m\frac{1}{|Z_G(c_i,h_i,d_i)(k)|}\cdot\inn{I_{x_n+u}}{f_{c_i,h_i,d_i}}.
\end{equation}
Using (\ref{eq:Xi}) for both $y=x$ and $y=x_n+u$, we have
\[
|\fu_P(k)|\cdot|\Xi_x(x)(k)|=\sum_{u\in\fu_P(k)}|\Xi_x(x_n+u)(k)|
\]
By definition of the Lusztig-Spaltenstein induction $N(x)$, the element $x_n+u$ is either in $N(x)$ (for $u$ in a dense subset of $\fu_P$) or is contained in an orbit of a smaller dimension. In the latter case, the orbit of $x_n+u$ must not meet $c_i+Z_{\fg}(d_i)$ for any nilpotent orbit meeting $c_i+Z_{\fg}(d_i)$ contains $c_i$ in its closure, i.e. $\Xi_x(x_n+u)=\emptyset$ when $x_n+u$ is in an orbit of a smaller dimension. For the other ``generic'' $u$ we have $\Xi_x(x_n+u)=\Theta(x_n+u)$ with $|\Theta(x_n+u)(k')|\sim|k'|^{\delta_c}$. Hence we have 
\[
|\fu_P(k')|\cdot|\Theta(x)(k')|\sim|k'|^{\delta_c+\dim\fu_P}\implies|\Theta(x)(k')|\sim|k'|^{\delta_c},
\]
which proves the proposition.
\end{proof}

\begin{proof}[Proof of Theorem \ref{thm:good1}] The centralizer $Z_G(x)$ acts on $\Theta(x)$. Since $\dim Z_G(x)=\dim Z_G(c)$ for $c\in N(x)$ \cite[Thm. 1.3(a)]{LS79}, any $Z_G(x)$-orbit on $\Theta(x)$ is a subvariety of dimension at least $\dim Z_G(x)-\dim Z_G(c,h,d)=\delta_c$ (for $(c,h,d)\in\Theta(x)$). By Proposition \ref{prop:1comp} there can be at most one such orbit, and the orbit has to be geometrically connected, i.e. it has to be a $Z_G(x)^o$-orbit. Hence we are done.
\end{proof}

We view Theorem \ref{thm:good1} as saying that any $x\in\fg(k)$ is ``supported'' by a nilpotent element in the following sense:

\begin{definition}\label{def:support} For $x\in\fg(k)$, $c\in\fg^{nil}(k)$, let us say $c$ {\bf supports} $x$ if we can complete it to an $\slt$-triple $(c,h,d)\in\Theta(x)(k)$. \end{definition}

Let us consider the homomorphism $\rho:\mu_m\ra\operatorname{Aut}(G)$ again and fix $a\in\Z/m$. The $G$-invariant form $B(\cdot,\cdot):\fg\times\fg\ra\mathbb{G}_a$ necessarily restricts to a perfect pairing $\fg_a\times\fg_{-a}\ra\mathbb{G}_a$. (Here $\mathbb{G}_a$ is the standard additive group scheme and has a priori nothing to do with $\fg_a$; we apologize for this funny notational clash.) Let $C(\fg_a(k))$ be the set of $\C$-valued functions on $\fg_a(k)$. For $f\in C(\fg_a(k))$ we put
\[
\hat{f}(y)=\sum_{x\in\fg_a(k)}\psi(B(x,y))f(x),\;\forall y\in\fg_{-a}(k)
\]
so that $f\mapsto\hat{f}$ defines a $G_0(k)$-equivariant isomorphism $C(\fg_a(k))\xra{\sim}C(\fg_{-a}(k))$. Let $J(\fg_a(k))\subset C(\fg_a(k))$ be the subspace of $G_0(k)$-invariant functions so that Fourier transform again gives $J(\fg_a(k))\xra{\sim}J(\fg_{-a}(k))$. Denote again by $J(\fg_a(k))^{nil}\subset J(\fg_a(k))$ the subspace of those supported on nilpotent elements and $J(\fg_a(k))^{conil}$ the image of $J(\fg_{-a}(k))^{nil}$ under Fourier transform (note the subscript $-a$).

\begin{proposition}\label{prop:restrict} The map $f\mapsto f|_{\fg_a^{nil}(k)}$ induces an isomorphism $J(\fg_a(k))^{conil}\xra{\sim} J(\fg_a(k))^{nil}$.
\end{proposition}

\begin{proof} Let $c_1,c_2,...$ be a set (apparently finite) of representatives of $G_0(k)$-orbits in $\fg_a^{nil}(k)$, ordered so that $\dim\Ad(G_0)c_i\le\dim \Ad(G_0)c_j$ for $i\le j$. We complete each $c_i$ into an $\slt$-triple $(c_i,h_i,d_i)$ with $h_i\in\fg_0$, $d_i\in\fg_{-a}$. Consider $I_{c_i+Z_{\fg_a}(d_i)}$ the function with value $1$ on $c_i+Z_{\fg_a(k)}d_i$ and $0$ elsewhere, and $f_{c_i,h_i,d_i}\in J(\fg_a(k))$ with $f_{c_i,h_i,d_i}(x)=\sum_{g\in G_0(k)}I_{c_i+Z_{\fg_a}}(\Ad(g)x)$. By the same proof as in Lemma \ref{lem:conil} we have $f_{c_i,h_i,d_i}\in J(\fg_a(k))^{conil}$.
	
By Lemma \ref{lem:retract}, 
$\supp(f_{c_i,h_i,d_i})\cap\fg_a^{nil}$ is the union of $\Ad(G_0(k))c_i$ and some $\Ad(G_0(k))c_j$ lying in larger $\Ad(G_0)$-orbits. In particular $\supp(f_{c_i,h_i,d_i})$ must be disjoint from $\Ad(G_0(k))c_i$ for $j<i$. This proves that $f_{c_i,h_i,d_i}|_{\fg_a^{nil}(k)}$ form a basis of $J(\fg_a(k))^{nil}$. In particular, the map $J(\fg_a(k))^{conil}\ra J(\fg_a(k))^{nil}$ given by restriction is surjective. Hence it is an isomorphism as both sides have the same dimension over the coefficient field $\C$.
\end{proof}

\begin{corollary}\label{cor:test} Suppose a function $f\in J(\fg_a(k))^{conil}$ takes values in $\R_{\ge 0}$. Suppose we have some $c\in\supp(f)\cap\fg_a^{nil}$, and that $c'\in\supp(f)\cap\fg_a^{nil}$ only if $c\in\Ad(G_0(k))c'$ or $c$ is contained in the Zariski boundary of $\Ad(G_0(\bar{k}))c'$. Then there exist
\begin{enumerate}
	\item Nilpotent elements $c_1,...,c_r$ such that $c$ is contained in the Zariski boundary of $\Ad(G_0(\bar{k}))c_i$ for each $1\le i\le r$. Each is completed into an $\slt$-triple $(c_i,h_i,d_i)$. We also complete $c$ to $(c,h,d)$. All these $\slt$-triples are defined over $k$.
	\item Constants $\alpha\in\R_{>0}$ and $\alpha_1,...,\alpha_r\in\R$ 
\end{enumerate}
such that
\[
f=\alpha\cdot f_{c,h,d}+\sum_{i=1}^r\alpha_i\cdot f_{c_i,h_i,d_i}.
\]
\end{corollary}

\begin{proof} This follows from the property established in the proof of Proposition \ref{prop:restrict} that $f_{c_i,h_i,d_i}|_{\fg_a^{nil}(k)}$ is supported at $\Ad(G_0(k))c_i$ and a number of $\Ad(G_0(k))$-orbits with strictly larger geometric orbits. 
\end{proof}

The point of the previous and the next corollary is that they give our template for test functions:

\begin{corollary}\label{cor:test2} Let $\ell\in\Z_{>0}$ and $\lambda:\mathbb{G}_m/_k\ra G_0$ be any cocharacter. Suppose $c\in{}^{\lambda}_{\ell}\fg_a$. Let $f_{\lambda,c}\in J(\fg_a(k))$ be given by $f_{\lambda,c}(x)=\sum_{g\in G_0(k)}I_{c+{}^{\lambda}_{<\ell}\fg_a}(\Ad(g)x)$. Then $f_{\lambda,c}$ satisfies the hypotheses of Corollary \ref{cor:test} and therefore the conclusions for the same $c$. 
\end{corollary}

\begin{proof} The Fourier transform of $I_{c+{}^{\lambda}_{<\ell}\fg_a}$ is supported on ${}^{\lambda}_{\le -\ell}\fg_{-a}$ and thus $f_{\lambda,c}\in J(\fg_a(k))^{conil}$. The rest follows from Lemma \ref{lem:test}.
\end{proof}

In the next corollary we go back to the case when $m=1$ and $\rho$ is trivial.

\begin{corollary}\label{cor:good1} Let $\ell\in\Z_{>0}$, $\lambda:\mathbb{G}_m/_k\ra G_0$ be any cocharacter, and $c\in{}^{\lambda}_{\ell}\fg(k)$. Suppose $c$ supports $x$ (Definition \ref{def:support}). Then $\Ad(G(k))x\cap(c+{}^{\lambda}_{<\ell}\fg)\not=\emptyset$. 
\end{corollary}

\begin{proof} Apply Corollary \ref{cor:test} to $\inn{I_x}{f_{\lambda,c}}$ where $f_{\lambda,c}$ is as in Corollary \ref{cor:test2} and get
\[
\inn{I_x}{f_{\lambda,c}}=\alpha\cdot\inn{I_x}{f_{c,h,d}}+\sum_{i=1}^r\alpha_i\cdot\inn{I_x}{f_{c_i,h_i,d_i}}.
\]
Since the orbit of $c_i$ is larger than that of $c$, we have $\dim\Ad(G)c_i>\dim\Ad(G)c=\dim\Ad(G)x$. Hence either $x$ is non-nilpotent and $\dim(\mathbb{G}_a\cdot\Ad(G)x)=\dim(\Ad(G)x)+1\le\dim\Ad(G)c_i$, or $x$ is nilpotent and $\dim(\mathbb{G}_a\cdot\Ad(G)x)=\dim\Ad(G)x\le\dim\Ad(G)c_i$. By Lemma \ref{lem:test} this gives $\inn{I_x}{f_{c_i,h_i,d_i}}=0$. At the same time to say $c$ supports $x$ is to say $\inn{I_x}{f_{c,h,d}}>0$. Hence $\inn{I_x}{f_{\lambda,c}}>0$ which proves the corollary.
\end{proof}

\p
\bibliographystyle{amsalpha}
\bibliography{biblio.bib}
	
\end{document}